\documentclass[reqno,12pt]{amsart}
\usepackage{amssymb,amsmath,latexsym,amsthm}
\usepackage{verbatim}
\numberwithin{equation}{section}
\usepackage{hyperref}

\usepackage[all]{xy}
\newcommand{\du}[2]{\langle#1,#2 \rangle}
\newcommand{\ip}[2]{( #1,#2)} 
\def\C{\mathbb{C}}

\def\K{\mathbb{K}}

\def\N{\mathbb{N}}

\def\R{\mathbb{R}}

\def\Z{\mathbb{Z}}

\def\cO{\mathcal{O}}

\def\cR{\mathcal{R}}
\def\cS{\mathcal{S}}

\def\Re{{\rm Re}\,}

\def\rank{{\rm rank}\,}

\theoremstyle{plain}
\newcommand{\gra}{\mathrm{gr\, }}

\newtheorem{theorem}[equation]{Theorem}

\newtheorem{lemma}[equation]{Lemma}

\newtheorem{proposition}[equation]{Proposition}

\newtheorem{definition}[equation]{Definition}

\newtheorem{remark}[equation]{Remark}

\newcommand{\fa}{\mathfrak{a}}

\newcommand{\fg}{\mathfrak{g}}
\newcommand{\fh}{\mathfrak{h}}
\newcommand{\fk}{\mathfrak{k}}

\newcommand{\fm}{\mathfrak{m}}
\newcommand{\fn}{\mathfrak{n}}

\newcommand{\fp}{\mathfrak{p}}
\newcommand{\fq}{\mathfrak{q}}
\newcommand{\fs}{\mathfrak{s}}
\newcommand{\fu}{\mathfrak{u}}
\newcommand{\fz}{\mathfrak{z}}

\newcommand{\bc}{\mathbf{c}}

\newcommand{\pr}{\mathrm{proj}}
\newcommand{\rr}{\mathrm{r}}

\newcommand{\SO}{\mathrm{SO}}
\newcommand{\SL}{\mathrm{SL}}
\newcommand{\Sp}{\mathrm{Sp}}

\newcommand{\SU}{\mathrm{SU}}
\newcommand{\rU}{\mathrm{U}}

\newcommand{\Hom}{\mathrm{Hom}}

\newcommand{\id}{\mathrm{id}}
\newcommand{\gL}{\Lambda }
\newcommand{\bK}{\mathbf{K}}
\newcommand{\bM}{\mathbf{M}}
\newcommand{\bX}{\mathbf{X}}
\newcommand{\bY}{\mathbf{Y}}
\newcommand{\bZ}{\mathbf{Z}}

\newcommand{\thetad}{\dot{\theta}}
\newcommand{\etad}{\dot{\eta}}

\newcommand{\sd}{\dot{\sigma}}

\newcommand{\bmid}{\,\, \right|\,\,}
\newcommand{\vin}{\varinjlim}

\def\sideremark#1{\ifvmode\leavevmode\fi\vadjust{\vbox to0pt{\vss
 \hbox to 0pt{\hskip\hsize\hskip1em
\vbox{\hsize2cm\tiny\raggedright\pretolerance10000 
 \noindent #1\hfill}\hss}\vbox to8pt{\vfil}\vss}}} 

\newcommand{\hf}{\widehat{f}}

\newcommand{\vv}{\varphi^\vee}

\newcommand{\G}{\Gamma}
\newcommand{\gG}{\mathrm{gr}\, \Gamma}

\newcommand{\gGm}[1]{(\mathrm{gr}\, \Gamma)_{#1}}

\newcommand{\kX}[1]{k_{\bZ}(#1 )}
\newcommand{\kXi}[1]{k_{\Xi} (#1)}
\newcommand{\Sph}[1]{\mathop{\mathrm{Sph}_{#1}}}
\newcommand{\alimit}{a\to \infty}

\newcommand{\Ad}{\mathrm{Ad}}

\begin{document}
\title[Radon transform for limits of symmetric spaces]{The Radon transform and its dual for limits of symmetric spaces}

\author{Joachim Hilgert}
\address{Institut f\"ur Mathematik, Universit\"at Paderborn, 33098 Paderborn, Germany}
\email{hilgert@upb.de}
\author{Gestur \'{O}lafsson}
\address{Department of Mathematics, Louisiana State University, Baton Rouge,
LA 70803, U.S.A.}
\email{olafsson@math.lsu.edu}
\thanks{The research of G. \'Olafsson was supported by  DMS-1101337}

\date{October 13, 2013}

\begin{abstract}
The Radon transform and its dual are central objects in geometric analysis on Riemannian symmetric spaces of the noncompact type. In this article we study algebraic versions of those transforms on inductive limits of symmetric spaces. In particular, we show that normalized versions exists on some spaces of regular functions on the limit. We give a formula for the normalized transform using integral kernels and relate them to limits of double fibration transforms on spheres.
\end{abstract}
\maketitle

\section{Introduction}
\noindent
Let $G_o$ be a classical noncompact connected semisimple Lie group and $G$ its complexification. We fix a Cartan involution
$\theta :G_o\to G_o$ on $G_o$ and denote the holomorphic extension to $G$ by the same letter. Let $K=G^\theta$ and let $K_o=K\cap G_o$ be the maximal compact subgroup corresponding
to $\theta$. Then $\bX=G_o/K_o$ is a Riemannian symmetric space of the noncompact type. The space $\bX$  is contained in its complexification $\bZ=G/K$. The subscript ${}_o$ will be used to denote subgroups in $G_o$. Dropping the index will then stand for the corresponding complexification in $G$.

Let $P_o=M_oA_oN_o$ be a minimal parabolic subgroup of $G_o$ with $A_o\subset \{a\in G_o\mid \theta (a)=a^{-1}\}$ and $M_o=Z_{K_o}(A_o)$. The space $\Xi_o=G_o/M_oN_o$ is the space of horocycles in $\bX$. We denote base point in $\bX$ by $x_o=\{K_o\}$ and the base point $\{M_oN_o\}$ in $\Xi_o$ by $\xi_o$. The (horospherical) Radon transform is the integral transform, initially defined on compactly supported functions on $\bX$, given by
\[\cR (f)(g\cdot \xi_o)=\int_{N_o} f(gn\cdot x_o)\, dn\]
for a certain normalization of the invariant measure $dn$ on $N_o$. The dual transform $\cR^*$ maps
continuous functions on $\Xi_o$ to continuous functions on $\bX$ and is given by
\[\cR^*(\varphi )(g\cdot x_o)=\int_{K_o} \varphi (gk\cdot \xi_o)\, dk,\]
where $dk$ denotes the invariant probability measure on $K_o$. If $f$ and $\varphi$ are compactly supported, then
\[\int_{\Xi} \cR(f)(\xi )\varphi (\xi)\, d\xi =\int_{\bX} f(x)\cR^*(\varphi )(x)\, dx\]
for suitable normalizations of the invariant measures on $\bX$, respectively $\Xi_o$. This explains why $\cR^*$ is called the dual Radon transform. For more detailed discussion we refer to Section \ref{se:RandDual}.

For a complex subgroup $L\subset G$ we call a holomorphic function $f\colon G/L\to \C$ regular if the orbit $G \cdot f$ with respect to the natural representation spans a finite dimensional subspace. We denote the $G$-space of regular functions by $\C[G/L]$. If $L_o\subset G_o$ is a subgroup such that $G_o/L_o$ can be viewed as a real subspace of its complexification $G/L$, then one calls a smooth function on $G_o/L_o$ regular, if its $G_o$-orbit spans a finite dimensional space. Since there is a bijection between regular functions on $G_o/L_o$ and $G/L$ we restrict our attention to $\C[G/L]$. The dual Radon transform can be extended to the space of regular functions on $\Xi$ but the integral defining the Radon transform is in general not defined for regular functions. In fact, a regular function on $\bZ$ can be $N_o$-invariant so the integral is infinite. This problem was first discussed in \cite{HPV02} and then further developed in \cite{HPV03}. Let us describe the main idea from \cite{HPV02} here. We refer to main body of the article for more details.

Denote the spherical representation of $G_o$ and $G$ with highest weight $\mu \in \fa_o^*$ by $(\mu_\mu,V_\mu)$, and its dual by $(\pi_\mu^*,V_\mu^*)$. The duality is written $\du{w}{\nu}$. Note that $(\pi_\mu,V_\mu)$ is unitary on a compact real form $U$ which we choose so that $U\cap G_o=K_o$. We fix a highest weight vector $u_\mu\in V_\mu$ of length one and a $K$-fixed vector $e_\mu^*\in V_\mu^*$ such that
$\du{u_\mu}{e_\mu^*}=1$. Fix a highest weight vector $u_\mu^*$ in $V_\mu^*$ such that $\du{u_\mu}{\pi_\mu^*(s_o)u_\mu^*}=1$, where $s_o\in K$ represents the longest Weyl group element. For $w\in V_\mu$ and $g\in G$ let
\[f_{w,\mu}(g\cdot x_o):=\du{w}{\pi_\mu^*(g)e_\mu^*}\quad \text{ and }\quad \psi_{w,\mu}(g\cdot \xi_o) :=\du{w}{\pi_\mu^*(g)u_\mu^*} \, . \]
Every regular function on $\bX$ is a finite linear combination of functions of the form $f_{w,\mu}$ and similarly for $\Xi$. The \textit{normalized Radon transform} on the space $\C[\bZ]$ of regular functions on $\bZ$ can now be defined by
\[\Gamma (f_{w,\mu}) := \psi_{w,\mu}\, .\]
The transform
\[\Gamma^{-1}(\psi_{w,\mu}):= f_{w,\mu}\]
defines a $G$-equivariant map $\C[\Xi]\to \C[\bX]$ which is inverse to $\Gamma$. Restricted to each $G$-type, the transform $\Gamma^{-1}$ is, up to a normalization given in Lemma \ref{le:cMu},  the dual Radon transform $\cR^*$. It is also shown in \cite{HPV03}, the dual Radon transform on $\C[\Xi]$ can be described as a limit of Radon transform over spheres, see Section \ref{se:RT-spheres} for details.

Our aim in this article is to study the normalized transforms $\Gamma$ and $\Gamma^{-1}$ as well as their unnormalized counterparts for certain inductive limits of symmetric spaces $\bX_j\subset \bZ_j$, called propagations of symmetric spaces, introduced in Section \ref{se:3.1}. This study is based on results from \cite{OW11a,OW11b} and \cite{DOW12} on inductive limits of spherical representations, which we use to study spaces of regular functions on the limit. More precisely, in Section \ref{se:RegInf} we consider two such spaces of regular functions, the projective limit $\varprojlim \C[\bZ_j]$ and the inductive limit $\C_i[\bZ_\infty]=\varinjlim \C[\bZ_j]$. The first main result is Theorem \ref{th:4.19} which describes how the graded version of $\Gamma$ extends to the projective limit.

We introduce the Radon transform and its dual in Section \ref{se:RT} and in Section \ref{se:RT-spheres} we recall the results from \cite{HPV03} about the Radon transform as a limit of a double fibration transform associated the spheres in $\bX$. In Section \ref{se:kernel} we show that the normalized Radon transform and its dual can be represented as an integral transform against kernel functions. Here the integral is taken over the compact group $U$. The corresponding result for the direct limit is Theorem \ref{th:KernelInf}.

Many of the results mentioned so far are valid for propagations of symmetric spaces of arbitrary rank, which means that they apply also to the case of infinite rank. For some results, however, we have to require that the rank of the symmetric space $\varinjlim X_j$ is finite. This is the case in particular in Section \ref{se:RInfty}, where we define the dual Radon transform $\cR^*$ for spaces of finite rank and connect it to the normalized dual Radon transform $\Gamma^{-1}$, see Theorem \ref{th:522}. Moreover, we define the Radon transform over spheres in this context, and connect it to the dual  Radon transform in Theorem \ref{th:525}.

{\sc acknowledgement.} We would like to thank E. B. Vinberg, who suggested to us that dual horospherical Radon transforms may exist also for limits of symmetric spaces.

\tableofcontents

\section{Finite dimensional geometry}\label{sec1}
\noindent
In this section we recall the necessary background from structure theory of finite dimensional symmetric spaces and related representation theory. Most of the material is in this section is standard, but we use this section also to set up the notation for later sections.

\subsection{Lie groups and symmetric spaces}
Lie group will always be denoted by uppercase Latin letters and their Lie algebra will be denoted by the corresponding lower case German letters. If $G$ and $H$ are Lie groups and $\theta :G\to H$ is a homomorphism, then the derived homomorphism is denoted by   $\thetad : \fg \to \fh$. If $G=H$, then
\[G^\theta = \{a\in G\mid \theta (a)=a\}\quad \text{and} \quad \fg^\theta =\{X\in\fg\mid \theta (X)=X\}\, .\]

From now on $G$ will stand for a connected simply connected complex semisimple Lie group with Lie algebra $\fg$. Let $U$ be a compact real form of $G$ with Lie algebra $\fu$ and let $\sd : \fg\to \fg$ denote the conjugation on $\fg$ with respect to $\fu$. We denote by $\sigma : G\to G$ the
corresponding involution on $G$. Then, as $G$ is simply connected, $U=G^\sigma$ and $U$ is simply connected.

Let $\theta : U\to U$ be a nontrivial involution and $K_o:=U^\theta$. Then $K_o$ is connected and $U/K_o$ is a simply connected symmetric
space of the compact type. Extend $\thetad :\fu\to \fu$ to a complex linear involution, also denoted by $\thetad$, on $\fg$. Denote
by $\theta : G\to G$ the holomorphic involution
with derivative $\thetad$. Write $\fu=\fk_o\oplus \fq_o$ where $\fk_o:= \fu^{\thetad}$ and $\fq_o:=\{X\in \fu \mid \thetad (X)=-X\}$. Let $\fs_o:= i\fq_o$ and $\fg_o:=\fk_o \oplus \fs_o$. Then $\fg_o$ is a semisimple real Lie algebra. Denote by  $G_o$ the analytic subgroup of $G$ with Lie algebra $\fg_o$. Then $G_o$ is $\theta$-stable, $G_o^\theta=K_o$, and $G_o/K_o$ is a symmetric space of the noncompact type. We have
$G_o=G^{\theta\sigma}$.

Let $K :=G^{\theta}$. Then $K_o=K\cap U=K\cap G_o$. Let $\bZ=G/K$, $\bX=G_o/K_o$, and $\bY=U/K_o$. As $\sigma $ and $\eta:=\sigma\theta$ map $K$ into itself it follows that both involutions define antiholomorphic involutions on $\bZ$ and we have
\[\bX=\bZ^\eta \quad \text{ and } \quad \quad \bY=\bZ^\sigma\, .\]
In particular $\bX$ and $\bY$ are transversal totally real submanifolds of $\bZ$.

If $V$ is a vector space over a field $\K$, then $V^*$ denotes the algebraic dual of $V$. If $V$ is a topological vector space, then the same notation will be used for the continuous linear forms. If $V$ is finite dimensional, then each $\alpha \in V^*$ is continuous.

Let $\fa_o$ be a maximal abelian subspace of $\fs_o$ and $\fa=\fa_o^\C$. For $\alpha \in \fa^*_o\subset \fa^*$ let
\[\fg_{o\alpha} :=\{X\in \fg_o\mid (\forall H \in \fa_o )\, [H,X]=\alpha (H)X\} \]
and
\[ \fg_\alpha :=\{X\in \fg\mid (\forall H \in \fa )\, [H,X]=
\alpha (H)X\} \, .\]
If $\alpha \not=0$ and $\fg_\alpha\not=\{0\}$, then $\fg_\alpha =\fg_{o\alpha} \oplus i\fg_{o\alpha}$ and $\fg_\alpha \cap \fu =\{0\}$. The linear form $\alpha\in\fa^*\setminus \{0\}$   is called a (restricted) root if $\fg_\alpha\not=\{0\}$. Denote by $\Sigma :=\Sigma (\fg,\fa)\subset \fa^*$ the set of roots.  Let $\Sigma_0 := \Sigma_0(\fg,\fa) := \{\alpha \in\Sigma\mid 2\alpha\not\in\Sigma\}$, the set of nonmultipliable roots. Then $\Sigma_0$ is a root system in the usual sense and the Weyl group $W$ corresponding to $\Sigma $ is the same as the Weyl group generated by the reflections $s_\alpha$, $\alpha \in \Sigma_{0}$. The Riemannian symmetric spaces $\bX$ and $\bY$ are irreducible if and only if the root system $\Sigma_{0}$ is irreducible.

Let $\fn:=\bigoplus_{\alpha\in\Sigma^+}\fg_\alpha$, $\fm:=\fz_{\fk}(\fa)=\{X\in\fk \mid [X,\fa ]=\{0\} \}$ and $ \fp:=\fm\oplus \fa\oplus \fn$.
All of those algebras are defined over $\R$ and the subscript ${}_o$ will indicate the intersection of those algebras with $\fg_o$. This intersection can also be described as the $\etad$ fixed points in the complex Lie algebra.

Define the parabolic subgroups $P_o:=N_{G_o}(\fp_o)\subset P:=N_G(\fp)$. We can write  $P_o=M_oA_oN_o$ (semidirect product) where $M_o:=Z_K(A_o)$, $A_o:=\exp \fa_o$, and $N_o:=\exp \fn_o$. Similarly we have  $P=MAN$. Let $F:=K\cap \exp i\fa_o$. Then each element of $F$ has order two and $M_o=F(M_o)^\circ$ where ${}^\circ$ denotes the connected component containing the identity element. We let $\Xi_o:=G_o/M_oN_o\subset \Xi:=G/MN$. As $\theta\sigma$ leaves $MN$ invariant it follows that $\Xi_o=\Xi^{\theta\sigma}$ is a totally real submanifold of $\Xi$.

Note than $K\cap MN =M$, so we obtain the following double fibration, which is of crucial importance for the Radon transforms:

\begin{equation}\label{Radonfib}
\xymatrix@M=7pt{
&G/M\ar [dl]_{p}\ar [dr]^{q}&\\
\bZ=G/K&&\Xi=G/MN}
\end{equation}

\subsection{Group spheres}

Let $\bX=G_o/K_o$ be as in the previous section. Denote by $s$ the symmetry of $\bX$ with respect to $x_o$. Then $\theta(g)=sgs^{-1}$ for $g\in G_o$.
Denote by $\mathrm X(A)$ the (additively written) group $\Hom (A_o,\R_+^\times)$ (where $\R_+^\times$ stands for the multiplicative group of positive numbers).  Then
$X (A)\simeq \fa^*$ where the isomorphism is given by $\mu \mapsto \chi_\mu$,
$\chi_\mu (a)=a^\mu $. We will simply write $\mu (a)$ for $\chi_\mu (a)$.

A {\it group sphere} in $\bX$ is an orbit of a maximal compact subgroup of $G_o$. Because of the Cartan decomposition $G_o=K_oA_oK_o$, and the
fact that all maximal compact subgroups in $G_o$ are $G_o$-conjugate,
any group sphere $\cS$ is of the form $gK_oa^{-1}\cdot
x_o=K_o^{g}ga^{-1}\cdot x_o$ with $g\in G_o$ and $a\in A_o$. The point $g\cdot
x_o$ is called the {\it center} of $\cS$ and $a$ is called the {\it radius} of $\cS$. The
group sphere of radius $a$ with center at $x$ is denoted by $\cS_a(x)$.

The group $G_o$ acts transitively  on the set $\Sph{a}\bX$ of group spheres of radius $a$.
The stabilizer of a group sphere $\cS:=\cS_a(g\cdot x_o)$ is a compact subgroup of $G$
containing the stabilizer $K_o^g$ of the point $g\cdot x_o$ and hence coinciding with it.
Since $g\cdot x_o$ is the only fixed point of $K_o^g$, it is uniquely determined by $\cS$.
Moreover, $\cS_{a_1}(x)=\cS_{a_2}(x)$ if and only if $a_1$ and $a_2$ are $W$-
equivalent.

We shall say that $a\in A_o$ tends to infinity, written $a\to\infty$, if $\alpha (a)\to \infty$ for any $\alpha \in \Sigma^+$.

The sphere
\begin{equation*}
\cS_a:=\cS_a(a\cdot x_o)=K^a \cdot x_o
\end{equation*}
of radius $a$ passes through $x_o$. It is known that it converges to the horosphere $\xi_o=N_o\cdot o$ as $a\to \infty$ (see, e.g., \cite{Ebe}, Proposition 2.6, and \cite{Ebe2}, p. 46).

Taking the sphere $\cS_a$ as the base point for the homogeneous space $\Sph{a}\bX$, we obtain the representation $\Sph{a}\bX=G_o/K_o^a$. This gives rise to the double fibration

\begin{equation}\label{sph-a-Radonfib}
\xymatrix@M=7pt{
&G_o/(K_o\cap K_o^a)\ar [dl]_{p_a}\ar [dr]^{q_a}&\\
\bX=G_o/K_o&&G_o/K_o^a=\Sph{a}\bX}
\end{equation}

Obviously, $K_o\cap K_o^a=Z_{K_o}(a)$. If $a$ is regular, then $Z_{K_o}(a)=Z_{K_o}(A_o)=M_o$. In this case the double fibration (\ref{sph-a-Radonfib}) reduces to

\begin{equation*}
\xymatrix@M=7pt{
&G_o/M_o\ar [dl]_{p} \ar [dr]^{q_a}& \\
\bX&&\Sph{a}\bX
}
\end{equation*}
 
\subsection{Spherical representations}
In this subsection we describe the set of spherical representations and the set of fundamental weights. Each irreducible finite dimensional representation $\pi$ of $U$ or $G_o$ extends uniquely to a holomorphic irreducible representation $\pi$ of $G$ and every irreducible holomorphic representation $\tau$ of $G$ is a holomorphic extension of an irreducible representation of $U$ and $G_o$. We will therefore concentrate on irreducible holomorphic respresentations of $G$. We will denote by $\pi^U$ respectively $\pi^o$ the restriction of a holomorphic representation $\pi$ to $U$ respectively $G_o$.

For a representation $\pi$ of a topological group $H$ or a Lie algebra $\fh$ we write $V_\pi$ for the vector space on which $\pi$ acts. Let
\[V^H_\pi=\{u\in V_\pi\mid (\forall k\in H)\,\, \pi (k)u=u\}\, .\]
Similarly
\[V^{\fh}_\pi=\{u\in V_\pi\mid (\forall X\in \fh)\,\, X u=0\}\, .\]
If $H$ is a connected Lie group with Lie algebra $\fh$ and $V_\pi$ a smooth representation of $H$, then $\fh$ acts on $V_\pi$ and $V^H_\pi=V^\fh_\pi$.

Back to our setup, as $K_o$ and $K$ are connected, it follows that if $\pi$ is a irreducible finite dimensional holomorphic representation of $G$ (and hence analytic), then
\[V_\pi^{K}:=V^{K_o}_\pi=V^\fk_\pi=V^{\fk_o}_\pi\, .\]
We say that $\pi$ is spherical if $V_\pi^{K}\not=\{0\}$ and that $\pi$ is conical if $V_\pi^{M N}\not=\{0\}$. Note that even if $M_o$ is not connected, then
\[V_\pi^{M N}=V_\pi^{M_oN_o}\, .\]
In fact, the inclusion $\subseteq$ is trivial and for the converse it suffices to note that $V_\pi^{M_oN_o}\subseteq V_\pi^{\fm_o+\fn_o}=:V_\pi^{\fm+\fn}$.

Define a representation $\pi^*$ on $V_\pi^*=V_{\pi^*}$ by
\[\du{v}{\pi^*(g)\nu}:=\du{\pi (g^{-1})v}{\nu}\, ,\quad g\in G, \, v\in V_\pi,\, \nu\in V_\pi^*\, .\]

For the following theorem see \cite{H94}, Thm. 4.12 and \cite{H84}, Thm. V.1.3 and Thm. V.4.1.

\begin{theorem} Let $\pi$ be an irreducible holomorphic representation of $G$. Then
the following holds:
\begin{enumerate}
\item
$\pi$ is spherical if and only if $\pi$ is conical. In that case
\[\dim V_\pi^{K}=\dim V_{\pi}^{M N}=1\, .\]
\item $\pi$ is spherical if and only if $\pi^*$ is spherical.
\end{enumerate}
\end{theorem}

Let
\begin{eqnarray}
 \Lambda^+(G,K)
&:=&\left\{ \mu \in i\fa^*_o\,\left|\, \frac{\ip{ \mu }{\alpha}}{\ip{ \alpha}{ \alpha}} \in \Z^+  \text{ for all }  \alpha \in \Sigma^+\right.\right\}\label{eq:L}\\
&=& \left\{ \mu \in i\fa^*_o\,\left|\, \frac{\ip{ \mu }{\alpha}}{\ip{ \alpha}{ \alpha}} \in \Z^+  \text{ for all }  \alpha \in \Sigma^+_0\right.\right\}\, .\nonumber
\end{eqnarray}
We mostly write $\gL^+$ for $\gL ^+(G,K)$. Let $W=N_{K_o} (\fa_o)/Z_{K_o}(\fa_o)$ denote the Weyl group. The parametrization of the spherical
representations is given by the following theorem.
 
\begin{theorem} \label{t-CH}  Let $\pi$ be
a  irreducible holomorphic representation of $G$, and $\mu$ its highest weight. Let $w_o\in W$ be such that $w_o \Sigma^+=-\Sigma^+$. Then the following are equivalent.
\begin{enumerate}
\item $\pi$ is spherical.
\item $\mu\in i\fa^*_o$ and $\mu \in \Lambda^+$.
\end{enumerate}
Furthermore, if $\pi$ is spherical with highest weight $\mu\in \gL^+$, then $\pi^*$ has highest weight $\mu^*:=-w_o\mu$.
\end{theorem}

\begin{proof} See \cite[Theorem 4.1, p. 535 and Exer. V.10]{H84} for the proof.
\end{proof}
 
If $\mu \in \gL^+$, then $\pi_\mu$ denotes the irreducible spherical representation with highest weight $\mu$.

Denote by $\Psi :=\{\alpha_1,\ldots ,\alpha_r\}$, $r:=\dim_\C \fa$, the set of simple roots in $\Sigma_{0}^+$. Define linear functionals $\omega_{j}\in i\fa^*_o$ by
\begin{equation}\label{fundclass1}
\frac{ \langle \omega_{i},\alpha_{j} \rangle } {\langle \alpha_{j},\alpha_{j} \rangle} = \delta_{i,j} \quad \text{ for }\quad 1 \leqq j \leqq r\ \ .
\end{equation}
Then $\omega_1,\ldots ,\omega_r\in\gL^+$ and
\[\gL^+=\Z^+\omega_1+\ldots  + \Z^+\omega_r= \left\{\left. \sum_{j=1}^r n_j\omega_j\bmid n_j\in \Z^+\right\}\, .\]
The weights $\omega_j$ are called the \textit{spherical fundamental weights for} $(\fg,\fk)$. Set $\Omega:=\{\omega_{1},\ldots ,\omega_{r}\}$.

\subsection{Regular functions}

Let $L$ be one of the groups $U$, $G_o$ and $G$. Let $\bM$ be a manifold and assume that $L$ acts transitively on $\bM$. Then $L$ acts on functions on $\bM$ by $a\cdot f(m)= f(a^{-1}\cdot m)$.  We say that $f\in C (\bM)$ is an $L$-\textit{regular function} if $\{a\cdot f\mid a\in L\}$ spans a finite dimensional space which we will denote by $\langle L\cdot f\rangle$. We denote by $\C_L[\bM ]$ the space of $L$-regular functions on $\bM$. Coming back to our usual notation we remark that the restriction map defines a $G_o$-isomorphism  $\C_G[\bZ]\to \C_{G_o}[\bX]$ and a $U$-isomorphism $\C_G[ \bZ]\to \C_U[\bY]$. Similarly, restriction defines a $G_o$-isomorphism $\C_G[\Xi]\to \C_{G_o}[\Xi_o]$. As soon as the acting group is clear from the context we will omit it from the notation.

We will mostly consider regular functions on the two complex spaces
$\bZ$ and $\Xi$. If needed, we will use results only stated or proved for the complex case also for the real cases using the above restriction maps.

For $\mu\in\gL^+$ we denote by $\C[\bZ]_\mu$, respectively $\C[\Xi]_\mu$, the space of regular functions on $\bZ$, respectively $\Xi$, of type $\pi_\mu$. We recall the following well known fact (cf. \cite{HPV02}):

\begin{lemma}\label{le:decC} The action of $G$ on $\C[\bX]_\mu$ and $\C[\Xi]_\mu$ is irreducible. As a $G$-module we have
\[\C [\bZ]=\bigoplus_{\mu\in\gL^+}\C [\bZ]_\mu\quad \text{and}\quad \C [\Xi]=\bigoplus_{\mu\in\gL^+}\C [\Xi]_\mu\, .\]
Each representation $\pi_\mu$ occurs with multiplicity one in each of those modules.
\end{lemma}

Let $f\in \C [\bX]_\mu$ be a highest weight vector. Recall that $KAN\subset G$ is open and dense. Let $kan\in KAN$. Then
\[(kan)\cdot f(x_o)= f(n^{-1}a^{-1}k^{-1}\cdot x_o) = a^{\mu}f(x_o)\]
where $a^\mu = e^{\du{\mu}{\log a}}$. Hence $f(x_o)\not= 0$. We denote by $f_\mu$\label{fmu1} the unique highest weight vector in $\C[\bX]_\mu$ with $f_\mu (e)=1$.

Let $s_o\in N_{K_o}(\fa_o)$ be a representative of the longest Weyl group element $w_o$ and recall that $Ns_oMAN$ is open and dense in $G$. Let $\psi \in \C[\Xi]_\mu$ be a highest weight vector. Then for $ns_o man_1\in Ns_oMAN$ we have
\[(ns_oman_1)\cdot \psi (\xi_o)= \psi (n_1^{-1}a^{-1}m^{-1}s^{-1}_on^{-1}\cdot \xi_o) =a^{\mu}\psi (s^{-1}_o\cdot \xi_o)\, .\]
Hence $\psi (s^{-1}_o\cdot \xi_o)\not= 0$.  Note that $s_o^2\in M$. As $\psi$ is $M$-invariant it follows that $\psi (s_o^{-1}\cdot \xi_o)=\psi (s_o \cdot \xi)$. Let $\psi_\mu $ be the unique highest weight vector in $\C[ \Xi]_\mu$ with $\psi_{\mu}(s_o\cdot \xi_o)=1$. According to Lemma~\ref{le:decC} there is a unique $G$-intertwining operator $\Gamma :\C[\bZ]\to \C[\Xi]$ such that $\Gamma (f_\mu )=\psi_\mu$ for all $\mu\in \gL^+$. For reasons which will become clear in section \ref{se:RT}, we call $\Gamma$ the \emph{normalized Radon transform}\label{Gamma-def} and note that its inverse $\Gamma^{-1} : \C[\Xi]\to \C[\bZ]$ is the unique $G$-isomorphism such that $\psi_\mu\mapsto f_\mu$ for all $\mu\in\gL^+$. Let
$\Gamma_\mu=\Gamma|_{\C[\bZ]_\mu}$. Then $\Gamma_\mu^{-1}=\Gamma^{-1}|_{\C[\Xi ]_\mu}$.

The maps $\Gamma_\mu$ and $\Gamma^{-1}_\mu$ have a simple description in terms of
the representation $(\pi_\mu,V_\mu)$. Fix for all $\nu\in\gL^+$ a $K$-fixed vector $e_{\nu}\in V_\nu$ and a highest weight vector $u_{\nu}$ in $V_{\nu}$. Further, choose the highest weight vector  $u_{\nu}^*\in V_{\nu}^*$ and the spherical vector  $e_\mu^*\in V_{\nu}^*$ according
to the normalization
$$\du{u_\nu}{\pi_{\nu}^*(s_o)u_{\nu}^*}=1\quad\text{and}\quad \du{u_\mu}{e_{\mu}^*}=1\, .$$
Then, for $v\in V_\mu$
\begin{equation}\label{de:fvarphi}
f_{v,\mu} (aK):=\du{v}{\pi_{\mu}^*(a)e_{\mu}^*}\quad \text{and}\quad  \psi_{v,\mu}(aMN) := \du{v}{\pi_{\mu}^*(a)u_{\mu}^*}\, .
\end{equation}
defines regular function $f_{v,\mu}$ on   $\bZ$, respectively $\psi_{v,\mu}$ on  $\Xi$. Furthermore,
\[V_\mu\ni v \mapsto f_{v,\mu}\in \C[\bZ]_\mu \quad \text{ and }\quad V_\mu \ni v \mapsto \psi_{v,\mu}\in \C[\Xi]_\mu\]
are $G$-isomorphisms. Note that  $f_\mu :=f_{u_\mu,\mu}$\label{fmu2} respectively $\psi_{\mu}:=
\psi_{u_\mu,\mu}$ are normalized highest weight vectors.

\section{Limits of symmetric spaces and spherical representations}
\noindent
In this section we introduce the notion of propagation of symmetric spaces and describe the construction of inductive limits of spherical representations from \cite{OW11a,OW11b}. We then recall the main result from \cite{DOW12} about the classification of spherical representations in the case where $U_{\infty}/K_{o\infty}$ has finite rank. 

We start with some facts and notations for limits of topological vector spaces, which will always be assumed to be complex, locally convex and Hausdorff. Similar notations for limits will be used for Lie groups and even sets without further comments. Our standard reference is  Appendix B in \cite{HY00} and the reference therein. 

If $W_1\subset W_2\subset \cdots $ is an injective sequence of vector spaces, then we denote the inclusion maps
$W_j\hookrightarrow W_k$, $k\ge j$, by $\iota_{k,j}$.  Let
\begin{equation}W_\infty:=\varinjlim W_j=\bigcup_{j=1}^\infty W_j
\end{equation}
and denote by $\iota_{\infty, j}$ the canonical inclusion $W_j\hookrightarrow W_\infty$. If each of the spaces $W_j$ is a topological vector space and each of the maps $\iota_{k,j}$ is continuous then a set $U\subseteq W_\infty$ is open in the \textit{inductive limit topology} on $W_\infty$ if and only if $U\cap W_j$ is open for all $j$. Then $W_\infty$ is a (again locally convex and Hausdorff) topological vector space. If $\{W_j\}$ and $\{V_j\}$ are inductive sequences of topological vector spaces and $T_j: W_j\to V_j$ is a family of continuous linear maps such that
\[\iota_{k,j}\circ T_j= T_k\circ \iota_{k,j}\]
where the first inclusion is the one related to the sequence $\{V_j\}$ and the second one
is the one associated to $\{W_j\}$. Then there exists a unique continuous linear map
$T_\infty =\varinjlim T_j : W_\infty \to V_\infty$ such that $\iota_{\infty,j}
\circ T_\infty = T_\infty\circ
\iota_{\infty, j}\circ T_\infty$ for all $j$.

If $W$ is a locally convex Hausdorff complex topological vector space, then $W^*$ will denote the space
of continuous linear maps $W\to \C$.  We provide it with the weak $*$-topology, i.e., the weakest
topology that makes all the maps $W^*\to \C$, $f\mapsto \langle x, f\rangle :=f(x)$, $x\in W$, continuous.
Then $W^*$ is also a locally convex Hausdorff topological vector space.
If $\{W_j\}$ is a inductive sequence of locally convex Hausdorff topological vector spaces then
$\{W_j^*\}$, with the projections $\pr_{j,k} : W_k^*\to W_j^*$, $\pr_{j,k}(\nu)
=\nu|_{W_j}$, $k\ge j$,  is a projective sequence of locally convex
Hausdorff topological vector spaces. Denote the projective limit of those spaces by
$\varprojlim W_j^*=W_\infty^*$. This notation is justified by the fact
that the topological dual of $W_\infty$ is $\varprojlim W_j^*$. We denote by
$\pr_{j,\infty} : W_\infty^* \to W_j$ the restriction map. If $\{W_j\}$ and $\{V_j\}$ are
injective sequences of topological vector spaces and $T_j: W_j\to V_j$ is as above, then
there exists a unique linear map $T_\infty ^*=\varprojlim T_j^* : V_\infty^*\to W_\infty^*$ such that
$\pr_{j,\infty}\circ T_\infty^* = T_j^*\circ \pr_{j,\infty}$ for all $j$. In fact
$T_\infty^*$ is just the adjoint of $T_\infty$.

We finish the subsection with a simple lemma that connects the inductive limit and the projective in case we have a injective sequence of Lie groups $G_j$ and $G_j$-modules $V_j$. This will be used several times later on. We leave the simple proof as an exercise for the reader.

\begin{lemma}\label{le:InjProLim} Let $\{G_j\}$ be an injective
sequence of Lie groups and  let $\{V_j\}$ be a projective sequence of
$G_j$-modules with $G_j$-equivariant projections
$\pr_{j,k}: V_k\to V_j$. Assume that we have $G_j$-equivariant inclusions
$\iota_{k,j}: V_j\to V_k$ making $\{V_j\}$ into an injective
sequence and such that $\pr_{j.k}\circ \iota_{k,j}=\id_{V_j}$.
For $f\in \varinjlim V_j$, fix
$j$ such that $f\in V_j$. Define $\iota_\infty (f):= \{\iota_{k+1,k}(f)\}_{k\ge j}$. Then $\varinjlim V_j$ and $\varprojlim V_j$ are $G_\infty$-modules and  $\iota$ is a well defined $G_\infty$-equivariant embedding $\varinjlim V_j \hookrightarrow \varprojlim V_j$.
\end{lemma}

\subsection{Propagation of symmetric spaces}\label{se:3.1}
\noindent
Assume that $G_1\subseteq G_2\subseteq \ldots \subseteq G_k\subseteq G_{k+1} \subseteq \ldots $ is a sequence of connected, simply connected
classical complex Lie groups as in the last section. In the following an index $k$ (respectively $j$) will always indicate objects related to $G_k$ (respectively $G_j$).
 We assume that $\theta_k|_{G_j}=\theta_j$ and $\sigma_k|_{G_j}=\sigma_j$ for all $j\le k$. Then $K_j=G_j\cap K_k$,
$U_j=G_j\cap U_k$, and $G_{jo}=G_j\cap G_{ko}$, for $k\ge j$.

This gives rise to an increasing sequence $\{\bZ_j=  G_j/K_j\}_{j\geqq 1}$ of simply connected complex
symmetric spaces such that for $k\ge j$ the embedding  $\bZ_j\hookrightarrow \bZ_k$ is a $G_j$-map. We denote this inclusion by $\iota_{k,j}$ and note that $\{\bZ_j\}$ is an injective system.

Similarly we have a sequence of transversal real forms $\bX_j=G_{jo}/K_{jo}$ and $\bY_j=U_j/K_{jo}$. We set
\[G_\infty :=\vin G_j \, ,\quad K_\infty : =\vin K_j\quad\text{and}\quad \bZ_\infty :=\vin \bZ_j=G_\infty /K_\infty\]
and similarly for other groups and symmetric spaces. Recall that, as a set
we have $\displaystyle G_\infty=\bigcup G_j$ but the
inductive limit  comes also with the inductive limit topology and a Lie group structure.  The space $\displaystyle \bZ_\infty =\bigcup \bZ_j$ is a smooth manifold
and the action of $G_\infty$ is smooth. Similar comments are valid for the other groups and the corresponding symmetric spaces.

In the following we will always assume that $k\ge j$ and $m\ge n$. As
$\theta_k|_{G_j}=\theta_j$ it follows that $\fk_k\cap \fg_j=\fk_j$ and
$\fs_k\cap \fg_j=\fs_j$. We choose the sequence $\{ \fa_j\}$ of maximal abelian subspaces of $\fs_j$ such that $\fa_k\cap \fs_j=\fa_j$.
Then $\Sigma_j\subseteq \Sigma_k|_{\fa_j}\setminus \{0\}$.
The ordering in $i\fa^*_o$ is chosen so that $\Sigma_j^+\subseteq \Sigma_k^+|_{\fa_{jo}}\setminus \{0\}$.

In case each $\bX_j$ is irreducible we say that $\bX_k$ \textit{propagates} $\bX_j$ if (i) $\fa_k=\fa_j$, or (ii) we obtain the Dynkin diagram for $\Psi_k$ is obtained from the Dynkin diagram for $\Psi_j$ by only adding simple roots at the left end (so the root $\alpha_1$ stays the same). Note, that usually the Dynkin diagram is labeled so that the first simple root is at the \textit{left} and. We have here reversed that labeling. Then, in particular, $\Psi_{k} =\{\alpha_{k,1},\ldots , \alpha_{k,r_k}\}$ and $\Psi_{j}=\{\alpha_{j,1},\ldots , \alpha_{j,r_j}\}$ are of the same type. Furthermore  $\alpha_{k,s}|_{\fa_j}=\alpha_{j,s}$ for $s=1,\ldots ,r_j$, see \cite{OW11a}. Furthermore, if $s\ge r_j+2$, then $\alpha_{k,s}|_{\fa_j}=0$.

In case of reducible symmetric spaces $\bX_t=\bX_t^1\times \cdots \bX_t^{s_t}$ we say that $\bX_k$ propagates $\bX_n$, $k\ge n$, $s_k\ge s_n$ and we can arrange the irreducible components so that $\bX_k^j$ propagates $\bX_n^j$ for $j=1,\ldots ,s_n$. We say that $\bZ_k$ propagates $\bZ_j$
if $\bZ_k$ propagates $\bZ_j$. {}From now on we will always assume, if nothing else is clearly stated, that the sequence $\{\bZ_j\}$ is so that $\bZ_k$ propagates $\bZ_j$ for $k\ge j$.

\subsection{Inductive limits of spherical representations}\label{sec3}
\noindent
In this section we recall the construction of inductive limits of spherical
representations from \cite{W09} and \cite{OW11b}.

As before we assume that $k\geqq j$ and that $\bZ_k$ propagates $\bZ_j$. Moreover,
from now on we will always assume that the groups $G_j$ are simple.
Denote by $\rr_{j,k}: \fa_k^*\to \fa_j^*$  the projection $\rr_{j,k}(\mu)=\mu|_{\fa_j}$. As shown in
\cite{OW11a,OW11b} we have

\begin{lemma} \label{resmult}
If $k\ge j$, then $\rr_{j,k}(\omega_{k,s})  =\omega_{j,s}$ for $s=1,\ldots ,r_j$.
\end{lemma}

This implies that the sets of highest weights $\Lambda^+_k:=\Lambda^+ (G_k,K_k)$ form a projective system with restrictions as projections. But those sets also form an injective system as we will now describe.  This will allow us to construct an injective system of representations in an unique way, starting at a given level $j_o$.

Let $\mu_j\in\gL_j^+$ and write
\[\mu_j=\sum_{s=1}^{r_j} k_s\omega_{j,s}\, , \qquad k_{s}\in \N_0\, .\]
Define $\mu_k\in \gL^+_k$ by
\[\mu_k:=\sum_{s=1}^{r_j}k_s \omega_{k,s}\, .\]
The map $\iota_{k,j}: \gL^+_j\to \gL^+_k$, $\mu_j\mapsto \mu_k$ is well defined and injective $\iota_{k,n}\circ\iota_{n,j}=\iota_{k,j}$ for $j\le n\le k$. We also have
\begin{equation}\label{eq:rr}
\rr_{j,k}\circ \iota_{k,j}=\id\, .
\end{equation}
Finally, $\mu_k $ is the minimal element in $\rr_{j,k}^{-1}(\mu_{j})$ with respect to the partial ordering $\nu - \mu =\sum_j k_j \omega_{k,j}$,  $k_j\in \Z^+$.
In particular we have the following lemma:
\begin{lemma}\label{le:iSw} The sequence $\{\gL^+_j\}_j$ with the maps
\[\iota_{k,j}: \gL_j^+\to \gL^+_k\, ,\quad \sum_{s=1}^{r_j}k_s\omega^j_s
\mapsto \sum_{s=1}^{r_j} k_s \omega^k_s\]
is an injective sequences of sets. Furthermore, there is a canonical inclusion
\[\gL^+_\infty:=\varinjlim \gL^+_j\hookrightarrow \varprojlim \gL_j^+\, .\]
\end{lemma}

\begin{proof} Most of the proof has been given already. For the last statement the idea is the same as in
Lemma \ref{le:InjProLim}. Given $j$ and $\mu_j\in \gL^+$. Then, by (\ref{eq:rr}) the sequence
$(\mu_j,\mu_{j+1},\ldots )$ is in $\varprojlim \gL^+_j$.
\end{proof}

For $j\in\N$ and $\mu=\mu_{j}\in \gL^+_{j}$ define $\mu_k = \iota_{k,j}(\mu)$, $k\ge j$,
and $\mu_\infty =\varinjlim \mu_j\in\gL^+_\infty$. Let $(\pi_{\mu_j}, V_{\mu_j})$ be the spherical representation
of $G_j$ with highest weight $\mu_j$. We can and will assume that each $V_{\mu_j}$ carries a $U_j$-invariant
inner product such that the embeddings $V_{\mu_j}\hookrightarrow V_{\mu_k}$ are isometric.

\begin{theorem}[\'O-W]\label{le:isr}
Assume that $\bZ_k$ propagates $\bZ_j$. Let $\mu_j\in\gL^+_j$ and define $\mu_k\in \gL^+_k$ as above. Then the following holds:
\begin{enumerate}
\item Let $u_{\mu_k}\in V_{\mu_k}$ be a weight vector
chosen as before, let $W_j:=\, <\pi_k (G_j)u_{\mu_k}>$ and  $\pi_{k,j}(g):=\pi_k(g)|_{W_j}$, $g\in G_j$. Then
$\pi_{k,j}$ is equivalent to $\pi_{\mu_j}$ and we can choose the highest weight
vector $u_{\mu_j}$ in $V_{\mu_j}$ so that the linear map generated by $\pi_{\mu_j}(g)u_{\mu_j}\mapsto \pi_{\mu_k}(g)u_{\mu_k}$ is a unitary $G_j$-isomorphism.
\item The multiplicity of $\pi_{\mu_j}$ in $\pi_{\mu_k}$ is one.
\end{enumerate}
\end{theorem}

\begin{remark} {\rm Note that in \cite{OW11a} the statement was proved
for the compact sequence $\{U_j\}$. But it holds true for the complex
groups $G_j$ by holomorphic extension. It is also true for the
non-compact groups $G_{j_o}$ by holomorphic extension and
then restriction to  $G_{jo}$.}\hfill $\qed$
\end{remark}

The second half of the above theorem implies that, up to a scalar, the only unitary $G_j$-isomorphism
is the one given in part (1). As a consequence we can and will always think of $V_{\mu_j}$ as a
subspace of $V_{\mu_k}$ such that the highest weight vector
$u$ is independent of $j$, i.e., $u_{\mu_j}=u_{\mu_k}$ for all $k$ and $j$. We form the inductive limit
\begin{equation}\label{de:Vinfty}
V_{\mu_{\infty}}:=\varinjlim V_{\mu_j}\, .
\end{equation}
Starting at a point $j_o$ the highest weight vector $u_{\mu_j}$, $j\ge j_o$ is constant and contained in
all $V_{\mu_j}$. I particular, $\mu_{\mu_j}\in V_{\mu_{\infty}}$. We also note that
$\{\pi_{\mu_j}(g)\}$, $g\in G_j$, forms a injective sequence of continuous
linear operators, unitary for $g\in U_j$. Hence $(\pi_{\mu_k})_\infty (g) : V_\infty \to V_\infty$ is a well defined continuous map. Similarly for the Lie algebra. We denote those maps by 
$\pi_{\mu_\infty}$ and $d\pi_{\mu_\infty}$ respectively. Hence group $G_\infty$ and
acts continuously, in fact smootly, on $V_{\mu_{\infty}}$.
We denote the corresponding representation of $G_\infty$ by $\pi_{\mu_\infty}$. We have \[ d\pi_{\mu_{\infty}}(H)u_{\mu_\infty}
=\mu_{\infty}(H)u_{\mu_\infty}\quad\text{ for all }\quad H\in\fa_\infty \,. \]
The representation $(\mu_{\mu_\infty},V_{\mu_\infty})$ is (algebraically) irreducible.
We can make $\pi_{\mu_\infty}|_{U_\infty}$ unitary by completing $V_{\mu_\infty}$ to an Hilbert space $\hat{V}_\infty$ as is usually done, see \cite{DOW12}.

The dual of $V_{\mu_\infty}$ is given by the corresponding projective limit
\begin{equation}\label{de:VinftyDual}
V_{\mu_\infty}^*=\varprojlim V_{\mu_j}^*\, .
\end{equation}
Note that in this notation $V_{\mu_\infty}^*\not=V_{\mu_\infty^*}$. We note that the highest weight vector, which we now denote by $u_{\mu_\infty}$
in $V_{\mu_\infty}$.
If $g\in G_k$, $k\ge j$, then $\pi_{\mu_j}^*(g)$ forms a projective family of
operators and hence $\varprojlim \pi_{\mu_k}^*$ is a well defined continuous
representation of $G_\infty$ on $V_{\mu_\infty}^*$. We
denote this representation by $\pi_{\mu_\infty}^*$. 

\begin{lemma}\label{le:PrlimKin}
Let the notation be as above. Then
\[\dim \left(\varprojlim V_{\mu_j}^{*}\right)^{K_\infty}=1\, .\]
\end{lemma}

\begin{proof} First fix $j_o$ so that $V_{\mu_j}$ is defined for
all $j\ge j_o$, i.e. $\{\mu_j\}$ stabilizes from $j_o$ on. As $\dim V_{\mu_j}^{*\, K_{j}}=1$, $j\ge j_o$, there exists a unique, up to scalar, $K_{j_o}$-fixed element $e_{\mu_{j_o}}^*$. We fix $e^*_{\mu_j}$ now so that $\pr_{j_o,j}(e_{\mu_j}^*) = e_{\mu_{j_o}^*}$, where $\pr_{j_0,j}$ is the dual map of $V_{\mu_{j_o}}\hookrightarrow V_{\mu_j}$. Then $e_{\mu_\infty}^*:=\{e_{\mu_j}^*\}_{j\ge j_o}\in  \left(\varprojlim V_{\mu_j}^*\right)^{K_\infty}$. On the other hand, if $\{e_{\mu_j}^*\}_{j\ge j_o}\in \left(  \varprojlim V_{\mu_j}^{*}\right)^{ K_\infty}$, then $e_{\mu_{j_o}}^*\in V_{\mu_{j_o}}^{*\, K_j}$ is unique up to scalar showing that the dimension is one.
\end{proof}

{}From now on we fix $e_{\mu_\infty}^*$ so that $\du{u_{\mu_\infty}}{e_{\mu_\infty}^*}=1$.

\begin{theorem} $V_{\mu_\infty}^*$ is irreducible.
\end{theorem}

\begin{proof} Assume that $W\subset V_{\mu_\infty}^*$ is a closed $G_\infty$-invariant subspace. Then $W^\perp =\{u\in V_{\mu_\infty}\mid (\forall \varphi \in W)\, \du{u}{\varphi}=0\}$ is closed and $G_\infty$ invariant. Hence $W^\perp=\{0\}$ or  $W=V_\infty$, and since all spaces involved are reflexive, this implies that $W=V_\infty^*$ or $W=\{0\}$.
\end{proof}

The vector $u_{\mu_\infty}  \in V_{\mu_\infty}$ is clearly $M_\infty N_\infty$-invariant. Therefore $V_{\mu_\infty}$ is conical (see \cite{DO13}).
But it is easy to see that with exception of some trivial
cases (as $G_j=G_k$ for all $j$ and $k$, which we do not consider) the representation
$(\pi_{\mu_\infty},V_{\mu_\infty})$ is not $K_\infty$-spherical. In fact, suppose that
$e\in V_{\mu_\infty}$ is a non-trivial, $K_\infty$-invariant vector. Let $j$ be so that $e\in V_{\mu_j}$. Then $e$ has to be fixed for all $K_s$, $s\ge j$ and hence a multiple of $e_s$. This is impossible in general as will follow from Lemma \ref{le:cMu}. On the other hand it was shown in \cite{DOW12} that the Hilbert space completion $\hat{V}_{\mu_\infty}$ is $K_\infty$-spherical if and only if the dimension of $\fa_\infty$ is finite. In this case we can assume that   $\fa_j=\fa_\infty$ for all $j$. Then
$\Sigma_j=\Sigma_k=\Sigma_\infty$, $\Sigma^+_j=\Sigma_k^+=\Sigma_\infty^+$ and
$\gL^+_j=\gL^+_k=\gL_\infty^+$ for $k\ge j$. But we still use the notation $\mu_j$ etc. to indicate what group we are using.

\begin{theorem}[\cite{DOW12}]\label{th:Ksph}  Let the notation be as above and assume that $\mu\not= 0$. Then $\hat V_\infty^{K_\infty}\not=\{0\}$ if and only if the ranks of the compact Riemannian symmetric spaces $\bX_k$ are bounded. Thus, in the case where $\bX_j$ is an irreducible classical symmetric space, we have $V_\infty^{K_\infty}\not=\{0\}$ only for $\SO(p+\infty)/\SO(p)\times \SO(\infty)$, $\SU(p+\infty)/\mathrm{S}(\rU(p)\times \rU(\infty))$ and
$\Sp(p+\infty)/\Sp(p)\times \Sp(\infty)$ where $0 < p < \infty$.
\end{theorem}

Let as usually $\iota_{k,j} :V_{\mu_j}\to V_{\mu_k}$ be the inclusion defined in Theorem \ref{le:isr} and
$\pr_{j,k} :V_{\mu_k} \to V_{\mu_j}$ the orthogonal projection. Then, as
$V_{\mu_j}\simeq <\pi_{\mu_k}(G_j)u_{\mu_k}>$, it follows that $\pr_{j,k}\circ \iota_{k,j}=
\id_{V_{\mu_j}}$. By Lemma \ref{le:InjProLim} there is a canonical $G_{\infty}$-inclusion
\[V_{\mu_\infty}\hookrightarrow \varprojlim V_{\mu_j}\, .\]

Define $u^*_{\mu_j}\in V_{\mu_j}^*$ by
\[\du{\pi_{\mu_j} (s_{o,j}) u_{\mu_j}}{u^*_{\mu_j}}=1\quad \text{ and }\quad u_{\mu_j}^*|_{(\pi_\mu (G_j) u_{\mu_j})^\perp_{\mu_j}}=0\]
where $s_{o,j}\in N_{K_j}(\fa_j)$ is so that $\Ad (s_{*,j})$ maps the set of positive roots into the set of negative roots and $(\pi_\mu (G_j) u_{\mu_j})^\perp_{\mu_j}$ is the orthogonal complement in $V_{\mu_j}$. We use the inner product to fix embeddings $V_{\mu_j}^*\hookrightarrow V_{\mu_k}^*$ for $j\le k$. Then again, we can take $u_{\mu_j}^*$ independent of $j$, which defines an $M_\infty N_\infty$-invariant element in $\varinjlim V_{\mu_j}^*\subset \varprojlim V_{\mu_j}^*=V_{\mu_\infty}^*$. As $e_{\mu_\infty}^*$ is $K_\infty$-invariant, it follows that
$V_{\mu_\infty}^*$ is both spherical and conical.

We now give another description of the representations $(\pi_{\mu_\infty},V_{\mu_\infty})$. This material is based on \cite{DO13}. We say that a representation $(\pi,V)$ of $G_\infty$ is \textit{holomorphic} if $\pi|_{G_j}$ is holomorphic for all $j$.

\begin{theorem}[\cite{DO13}]\label{th:DO13}
Assume that $X_\infty$ has finite rank. If $\mu_\infty\in \gL^+_\infty = \gL^+$, then $(\pi_{\mu_\infty},V_{\mu_\infty})$ is irreducible, conical
and holomorphic. Conversely, if $(\pi,V)$ is an irreducible conical and holomorphic representation of $G_\infty$, then there exists a unique $\mu_\infty\in \gL^+_\infty$ such that $(\pi,V)\simeq (\pi_{\mu_\infty},V_{\mu_\infty})$.
\end{theorem}

\section{Regular functions on limit spaces}\label{se:RegInf}
\noindent
In this section we study the spaces $\varinjlim \C[\bZ_j]$ and $\varprojlim \C[\bZ_j]$ as well as their analogs for $\Xi_\infty$.
Our main discussion
centers around the injective limits. We only discuss the limits of the complex cases. The corresponding results  for the algebras 
$\C_i [\bX_\infty]=\varinjlim \C[\bX_j]$ and $\C_i[\bY_\infty]:=\varinjlim \C[\bY_\infty]$ can be derived simply by restricting functions from $\bZ_
\infty$ to the real subspaces $\bX_\infty$ respectively $\bY_\infty$.

\subsection{Regular functions on $\bZ_\infty$}
In this section $\{\bZ_j\}=\{G_j/K_j\}$ is a propagated system of symmetric spaces as before. There are two natural ways to extend the notion of a regular function on finite dimensional symmetric spaces to the inductive limit of those spaces. One is to consider the projective limit $\C [\bZ_\infty]:=\varprojlim \C [\bZ_j]$. The other possible generalization would be to consider the space of functions on $\bZ_\infty$ which are algebraic finite sums of algebraically irreducible $G_\infty$-modules and such that each $f$ is locally finite in the sense that for each $j$ the space $\langle G_j\cdot f\rangle$ is finite dimensional. But as very little is known about those spaces and not all of our previous discussion about the  Radon transform and its dual generalize to those spaces we consider first the space
\[\C_i[\bZ_\infty]:=\varinjlim \C[\bZ_j]\, .\]
That this limit in fact exists will
be shown in a moment.

Let $x_\infty= \{K_\infty\}\in \bZ_\infty$ be the base point
in $\bZ_\infty$. Then all the spaces $\bZ_j$ embeds into $\bZ_\infty$ via
$aK_j\mapsto a\cdot x_\infty$ and in that way $\bZ_\infty =\bigcup \bZ_j$. Recall from our previous discussion and Lemma  \ref{le:iSw} that the sets $\gL^+_j$ of highest spherical weights form an injective system and $\gL^+_\infty=\varinjlim \gL^+_j$.
Each $\mu_\infty=\varinjlim \mu_j$ determines a unique algebraically irreducible (see below for proof) $G_\infty$-module $V_{\mu_\infty}=\varinjlim V_{\mu_k}$ such that the  dual space $V_{\mu_\infty}^*=\varprojlim V_{\mu_j}^*$ contains a (normalized) $K_\infty$-fixed vector $e_{\mu_\infty}^*$ normalized by the condition
$\du{u_{\mu_\infty}}{e_{\mu_\infty}^*}=1$ as after Lemma \ref{le:PrlimKin}.  As before, we
denote the $G_\infty$-representation on $V_{\mu_\infty}^*$ by $\pi_{\mu_\infty}^*$ and
consider the $G_\infty$-map
\[ V_{\mu_\infty}\hookrightarrow \text{ space of continuous functions on } G_\infty \]
given by
\begin{equation}\label{def:fvInf}
w\mapsto f_{w,\mu_\infty}\, ,\quad \text{ where } f_{w,\mu_\infty}(a\cdot x_\infty):=\du{w}{\pi_{\mu_\infty}^*(a)e_{\mu_\infty}^*}\, .
\end{equation}
Denote the image of the map (\ref{def:fvInf}) by $\C [\bZ_\infty]_{\mu_\infty}$. Thus
\begin{equation}\label{def:CZInfty}
\C[\bZ_\infty]_{\mu_\infty}=\{f_{w,\mu_\infty}\mid w\in V_{\mu_\infty}\}\simeq V_{\mu_\infty}\, .
\end{equation}

Note that the restriction of (\ref{def:fvInf}) to $V_{\mu_j}$ and $\bZ_j$ is the $G_j$-map
\[f_{w,\mu_j}(x)=\du{w}{\pi_{\mu_j}^*(a)e_{\mu_j}^*}= \du{w}{\pi_{\mu_j}^*(a)e_{\mu_\infty}^*}\, ,\quad x=a\cdot x_\infty\, , \]
introduced in (\ref{de:fvarphi}). That this is possible
follows from the proof of Lemma \ref{le:PrlimKin}. Hence we have a canonical $G_j$-map $\C[\bZ_j]_{\mu_j}\hookrightarrow \C[\bZ_k]_{\mu_k}$ for $k\ge j$ such that the following diagram commutes:

\begin{equation*}
\xymatrix{
V_{\mu_j} \ar[d] \ar[r] &
 V_{\mu_k} \ar[d]\ar[r] & \ldots  \\
\C[\bZ_j]_{\mu_j}\ar[r] &  \C[\bZ_k]_{\mu_k}\ar[r] &\ldots
}
\end{equation*}

As $\C[\bZ_j]=\sum^\oplus \C[\bZ_j]_{\mu_j}$ and  $\C[\bZ_k]=\sum^\oplus \C[\bZ_k]_{\mu_k}$ one derives that the spaces $\C[\bZ_j]$ form an injective system. Note that $\varinjlim \C[\bZ_j]_{\mu_j}$ and $\C_i [\bZ_\infty]:=\varinjlim \C[\bZ_j]$ carry natural $G_\infty$-module structures. This proves part of the following theorem:

\begin{theorem} The space $\C [\bZ_\infty]_{\mu_\infty}$ is an algebraically irreducible $G_\infty$-module, and
\[\C [\bZ_\infty]_{\mu_\infty} =\varinjlim \C[\bZ_j]_{\mu_j}\simeq V_{\mu_\infty}\, .\]
Furthermore
\[\C_i [\bZ_\infty] ={\sum_{\mu_\infty\in\gL_\infty^+}\!\!}^\oplus \C[\bZ_\infty]_{\mu_\infty}\]
as a $G_\infty$-module
\end{theorem}

\begin{proof} (See also \cite[Thm. 1]{KS77} and \cite[\S 1.17]{O90}.) Everything is clear except maybe the irreducibility statement. For that it is enough to show that $V_{\mu_\infty}$ is algebraically irreducible. So let $W\subset V_{\mu_\infty}$ be $G_\infty$-invariant. If $W\not=\{0\}$, then we must have $W\cap V_{\mu_j}\not= \{0\}$ for some $j$. But then $W\cap V_{\mu_k}\not=\{0\}$ for all $k\ge j$ and $W\cap V_{\mu_k}$ is $G_k$-invariant. As $V_{\mu_k}$ is algebraically irreducible it follows
that $V_{\mu_k}\subset W$ for all $k\ge j$. This finally implies that $W=V_{\mu_\infty}$.
\end{proof}

\begin{remark}\label{rem:fin-rank} {\rm In the case where the real infinite dimensional space $\bX_\infty$ has finite rank the space $\C_i[\bZ_\infty]$ has a nice representation theoretic description. In this case,
as mentioned earlier, we may assume $\gL^+_\infty=\gL^+_j$. We have also noted that each of the spaces $V_{\mu_j}$ is a unitary representation
of $U_j$ such that the embedding $V_{\mu_j}\hookrightarrow V_{\mu_k}$ is a $G_j$-equivariant isometry and the highest weight vector $u_{\mu_j}$ gets mapped to the highest weight vector $u_{\mu_k}$. In that way we have $u_{\mu_\infty}=u_{\mu_j}$ for all $j$. Furthermore this leads to a pre-Hilbert structure on $V_{\mu_\infty}$ so that $V_{\mu_\infty}$ can be completed to a unitary irreducible $K_\infty$-spherical representation $\hat V_{\mu_\infty}$ of $G_\infty$ (see \cite[Thm. 4.5]{DOW12}). Furthermore, it is shown in \cite{DO13} that each unitary $K_\infty$-spherical representation $(\pi,W_\pi)$ of $G_\infty$ such that $\pi|_{U_j}$ extends to a holomorphic representation of $G_j$ for each $j$ is locally finite and of the form $\hat{V}_{\mu_\infty}$ for some $\mu_\infty\in\gL^+$. Moreover, each of those representations is conical in the sense that $\hat V_{\mu_\infty}^{M_\infty N_\infty} \not=\{0\}$. Finally, each irreducible unitary conical representation $(\pi,W_\pi)$ of $G_\infty$, whose restriction to $U_\infty$ whose restriction to $U_j$ extends to a holomorphic representation of $G_j$ is unitarily equivalent to some $\hat V_{\mu_\infty}$. }\hfill $\qed$
\end{remark}

The inclusions $\iota_{k,j} :  \bZ_j \hookrightarrow \bZ_k$ lead to projections on the spaces of functions given by restriction. In particular, we have the projections
\begin{equation}
 \pr_{j,k} : \C[\bZ_k]\to \C[\bZ_j]
\end{equation}
satisfying $\pr_{j,n}\circ \pr_{n,k}=\pr_{j,k}$. Hence  $\{\C [\bZ_j]\}$ is a projective sequence  and $ \varprojlim \C [\bZ_j]$ is a $G_\infty$-module, and in fact an algebra, of functions on $\bZ_\infty$. In fact, let $f=\{f_j\}_{j\ge j_o}\in \varprojlim \C [\bZ_j] $ and $x\in \bZ_\infty$. Let $j$ be so that $x\in\bZ_j$. Define $f(x):=f_j(x)$. If $k\ge j$, then $ \pr_{j,k}(f_k)= f|_{\bZ_j}$. In particular $f_k(x)=f_j(x)$. Hence $f(x)$ is well defined.

In general we do not have $\pr_{j,k} (\C[\bZ_k]_{\mu_k})\subset \C[\bZ_j]_{\mu_j}$, but
\[\pr_{j,k}\circ \iota_{k,j}|_{\C[\bZ_j]_{\mu_j} }= \mathrm{id}_{\C[\bZ_j]_{\mu_j}}\]
as this is  satisfied on level of representations $V_{\mu_j}\to V_{\mu_k}\to V_{\mu_j}$ as
mentioned before. Hence.  by Lemma \ref{le:InjProLim}, we can view $\varinjlim \C[\bZ_\infty]_{\mu_j}$ as a submodule of $\varprojlim \C[\bZ_j]_{\mu_j}$. We record the following lemma which is obvious from the above discussion:

\begin{lemma} We have a $G_\infty$-equivariant embedding
\[\C_i [\bZ_\infty] \hookrightarrow \varprojlim \C[ \bZ_j]\, .\]
\end{lemma}

\subsection{Regular functions on $\Xi_\infty$}\label{se:RegfctsHor} 
\noindent
In order to construct regular functions on $\Xi_\infty$ we apply the same construction to the horospherical spaces $\Xi_j$ we used for the symmetric spaces $\bZ_j$. As the arguments are basically the same, we often just state the results.

The following can easily be proved for at least some examples of infinite rank symmetric spaces like $\SL (j,\C)/\SO (j,\C)$, but we only have a general prove in the obvious case of finite rank.

\begin{lemma} Assume that the rank of $\bZ_j$ is constant. Then for $k\ge j$ we have
\[M_j=M_k\cap G_j\quad \text{ and }  \quad N_j=N_k\cap G_j\, .\]
\end{lemma}

\begin{definition}{\rm
We say that the injective system of propagated symmetric spaces $\bZ_j$ is \emph{admissible} if $M_j=M_k\cap G_j$ and $N_j=G_j\cap N_k$ for all $k\ge j$.
}\end{definition}

{}From now on we will always assume that the sequence  $\{\bZ_j\}$ of symmetric spaces is admissible. Let $\xi_o=eM_\infty N_\infty$ be the base point of $\Xi_\infty$ and note that we can view this as the base point in $\Xi_j\simeq G_j\cdot \xi_o\subset \Xi_\infty$.

For $\mu_\infty= \varinjlim \mu_j\in\gL^+$,  $w\in V_{\mu_j}$,  $\xi=a\cdot \xi_o\in \Xi_j$, $a\in G_j$ we have
\[\psi_{w,\mu_j}(\xi )=\du{w}{\pi_{\mu_j}^*(a)u_{\mu_j}^*}= \du{w}{\pi_{\mu_j}^*(a)u_{\mu_\infty}^*}=:\psi_{w,\mu_\infty} (\xi )\, .\]
This defines $G_j$-equivariant inclusions
\[\C[\Xi_j]_{\mu_j}\hookrightarrow
\C[\Xi_k]_{\mu_k}\hookrightarrow \C[\Xi_{\infty}]_{\mu_\infty}:=\{\psi_{w,\mu_\infty}\mid w\in V_{\mu_\infty}\}\simeq
\varinjlim \C[\Xi_j]_{\mu_j} \, .
\]
We note that $V_{\mu_\infty}\to \C[\Xi_\infty]_{\mu_\infty}$, $w\mapsto \psi_{w,\pi_\infty}$, is a $G_\infty$-isomorphism. With the same argument as above this leads to an injective sequence
\[\C[\Xi_j] \hookrightarrow\C[\Xi_k] \hookrightarrow \varinjlim \C[\Xi_j]=:\C_i [\Xi_\infty] \, . \]

\begin{theorem}\label{th:CXiInf} Assume that the sequence $\{\bZ_j\}$ is admissible. Then
\[\C_i[\Xi_\infty]={\sum_{\mu_\infty\in\gL^+_\infty}\!\!\!}^\oplus \C [\Xi_\infty]_{\mu_\infty}\, .\]
\end{theorem}

Denote by $\Gamma_j$ the normalized Radon transform $\Gamma_j :\C[\bX_j]=\C[\bZ_j]\to \C [\Xi_j]$ introduced in Section~\ref{Gamma-def} and set
\begin{equation}\label{de:Rinfty}
\G_\infty (f_{v,\mu_\infty}  ):= \psi_{v,\mu_\infty}\, .
\end{equation}
Then $\G_\infty$ defines a $G_\infty$-equivariant map $\G_\infty : \C_i[\bZ_\infty]\to \C_i [ \Xi_\infty]$ and
\[\Gamma_\infty =\varinjlim \Gamma_j\,.\]
As each $\Gamma_j$ is invertible it follows that $\Gamma_\infty$ is also invertible. In fact, the inverse is $\Gamma_\infty^{-1}=\varinjlim \Gamma_{j}^{-1}$ which maps $\psi_{v,\mu_\infty}$ to $f_{v,\mu_\infty}$. As a consequence we obtain the following theorem:

\begin{theorem}\label{thm:Gamma} Suppose that the sequence $\{\bZ_j\}$ is admissible. Let $\mu_\infty=\varinjlim \mu_j \in \Lambda^+_\infty$, $k\ge j$, and  $w\in V_{\mu_k}$. Then
\[\iota_{k,j}\circ \Gamma_j f_{w,\mu_j} = \Gamma_k(\iota_{k,j}\circ f_{w,\mu_j})=
\psi_{w,\mu_\infty}\]
and
\[\iota_{k,j}\circ \Gamma_j^{-1} \psi_{w,\mu_j} =
\Gamma_k^{-1}(\iota_{k,j}\circ \psi_{w,\mu_j})=
f_{w,\mu_\infty}\, .\]
In particular, we have the following commutative diagram: 

\begin{equation*}
\xymatrix{\cdots \, \ar[r]&\C[\bZ_j] \ar[r]^{\iota_{k,j}} \ar[d]_{\G_j}&
\C[\bZ_k]\ar[r]^{\iota_{k,\infty}}\ar[d]_{\G_k}&
 \C_i[\bZ_\infty]\ar[d]_{\G_\infty} \\
\cdots \, \ar[r]
&\C[\Xi_j]\ar[r]^{\iota_{k,j}}\ar[u]_{\G_j^{-1}} &\C[\Xi_k]\ar[r]^{\iota_{\infty,k}}
\ar[u]_{\G_k^{-1}}
 &
\C_i[\Xi_\infty]\ar[u]_{\G_\infty^{-1}}
}
\end{equation*}
\end{theorem}

\subsection{The projective limit}
We  discuss the projective limit in more detail. First we need the following notation.
For $\mu,\nu\in \fa_o^*$ write
\begin{equation}\label{de:le}
\nu \le \mu\quad  \text{if}\quad \mu - \nu=\sum_{\alpha\in\Sigma^+} n_\alpha \alpha\, ,\quad n_\alpha\in \N_0\,.
\end{equation}
If $\nu \le \mu$ and $\nu\not= \mu$ then we also write $\nu <\mu$.
The main problem in studying the projective limit is the decomposition of
$V_{\mu_k}|_{G_j}$ for $k\ge j$. It is not clear if these
representations  decompose into representations with highest weight $\nu_j\le \mu_j$. In case the rank of
$\bX_\infty$ is finite that is correct. We therefore in the reminding of this subsection  make the assumption
that the rank of $\bX_\infty$ is finite.
In this case we can--and will--assume that $\fa_j=\fa$ for all $j$ and recall from earlier discussion that
$\Sigma_j=\Sigma$ is constant and so are the sets of positive roots $\Sigma_j^+=\Sigma^+$ and the
sets of
highest weights $\Lambda_j^+=\Lambda^+$.

Write
\begin{equation}\label{eq:rest}
(\pi_{\mu_k},V_{\mu_k})|_{G_j}\simeq \bigoplus_{s=0}^r(\pi_{s},W_{s})
\end{equation}
with $(\pi_0,W_{0})\simeq (\pi_{\mu_j},V_{\mu_j})$ which occurs with multiplicity one.

\begin{lemma} Assume that the rank of $\bX_\infty$ is finite. Let $\mu\in \gL^+$. Then  we have the following
\begin{enumerate}
\item  $e_{\mu_k}^*|_{W_s}=0$ if $W_s$ is not spherical.
\item Assume that $W_{s}\simeq V_{\nu}$ is spherical. Then $\nu < \mu_j$.
\end{enumerate}
\end{lemma}

\begin{proof} The first claim is obvious. Let $\sigma$ be a weight of $V_{\mu_k}$.
\[\sigma  = \mu_k - \sum_{\alpha\in\Psi_k}  t_\alpha \alpha\]
with $t_\alpha$ nonnegative integers. This in particular holds if $\sigma$ is a highest weight of a spherical representation of $G_j$ proving the claim.
\end{proof}

\begin{lemma} Assume that the rank of $\bX_\infty$ is finite. Let $k>j$ and let $v\in V_{\mu_k}$.   Then
\[\psi_{v,\mu_\infty}|_{\Xi_j}=\psi_{v,\mu_k}|_{\Xi_j}=\psi_{\pr_{j,k}(v),\mu_j}\, .\]
\end{lemma}

\begin{proof} Let $x\in G_j$. Then $\pi_{\mu}^* (x)u_{\mu_k}^*\in \langle \pi_{\mu_k} ( G_j) u_{\mu_k}^*
\rangle = V_{\mu_j^*}=V_{\mu_j}^*$. The claim now follows as $u_{\mu_j}^*=u_{\mu_k}^*$.
\end{proof}

We finish this section by recalling the graded version of $\C[\bZ_j]$ and $\Gamma$. Recall that
we are assuming that the rank of $\bX_j$ is finite.
Note that even if $\{\C [\bZ_j]\}$ is a projective sequence, the sequence $\{\C[\bZ_j]_{\mu_j}\}$ is not projective in general.
Consider the ordering on $\fa_o^*$ as above.  This   defines a filtration on $\C [\bZ_j]$ and we denote by $\gra \C [\bZ_j]$ the corresponding graded module. Thus
\begin{equation}\label{de:ga}\gra \C [\bZ_j ]_{\mu_j} =\bigoplus_{\nu_j \le \mu_j} \C [\bZ_j ]_{\nu_j} / \bigoplus_{\nu_j <\mu_j} \C [\bZ_j ]_{\nu_j}
\end{equation}
and
\begin{equation}\label{de:ga1}
\gra \C [\bZ_j]\simeq_{G} \bigoplus_{\mu\in \gL^+_j} \gra \C [\bZ_j ]_{\mu_j}\, .
\end{equation}

If $f\in \C[\bZ_j]$ then $[f]$ denotes the class of $f$ in $\gra \C[\bZ_j]$. Let $\kappa_j$, respectively $\kappa_{\mu_j}$, be the $G$-isomorphism $\C [\bZ_j]\simeq \gra \C [\bZ_j]$, respectively $\C [\bZ_j]_{\mu_j}\simeq \gra \C [\bZ_j]_{\mu_j}$, given by $f\mapsto [f]$. We let $\gG_j:=
\Gamma_j\circ \kappa^{-1}_j$ and $\gGm{\mu_j}=\Gamma_j\circ \kappa_{\mu_j}^{-1}$. Note that
this construction is also valid for $j=\infty$.

\begin{proposition}[Prop. 7 \cite{HPV02}] The $G_j$-map $\gG_j$ is an ring isomorphism $\gra \C[\bZ_j]\to \C[\Xi_j]$.
\end{proposition}

In the following we will not distinguish between $\Gamma_j$ and $\gG_j$ except where necessary.  Thus we will prove statements for $\Gamma$ and then use it for $\gG$ without any further comments.

Identifying functions on $G_k/K_k$ with right $K_k$-invariant
functions on $G_k$ the following is clear
\[\pr_{j,k}(f_{v,\mu_k}) =\du{v}{\pi_{\mu_k}^*|_{G_k} e_{\mu_k}^*}\, .\]
Thus, the  kernel of $\pr_{j,k}$ is the $G_j$-module
\[\ker \pr_{j,k}=\{v\in V_k\mid v\perp  \pi_k(G_j)e_{\mu_k}^* \}\, .\]
As $\du{v_{\mu}}{e_{\mu_k}^*}\not= 0$ it follows that
$\ker \pr_{j,k}$ is a sum of $G_j$-modules with highest weight $< \mu_j$. Hence
\[\gra \pr_{j,k} : \gra \C[\bZ_k]\to \gra \C [\bZ_j]\]
is well defined and
\[ \gra \pr_{j,k} (\C [\bZ_k]_{\mu_k}) = \gra \C[\bZ_j]_{\mu_j}\, .\]
It follows  that the sequence $\{\gra \C[\bZ_j]_{\mu_j}\}$ is
projective.

We can also form the graded algebra $\gra \C_i[\bZ_\infty]$ as in
(\ref{de:ga}) and (\ref{de:ga1}).  Again we can view elements in $\gra \C_i[\bZ_\infty]$
as functions on $\bZ_\infty$ by choosing the unique element in $g\in [f ]\in
\gra\C_i [\bZ_\infty]_\mu$ so that $g \in \C_i[\bZ_\infty]_\mu$.
The inclusion
\[\gra\iota_{k,j} :\C[\bZ_j]=  \sum^{\oplus} \gra\C[\bZ_j]_{\mu_j}
\to \gra \C[\bZ_k]_{\mu_k}=\C[\bZ_k] \]
given by
\[ \sum [ f_{v_j,\mu_j}]\to \sum [f_{\iota_{k,j}(v_j),\mu_j}]\]
satisfies the
relation $\gra \pr_{j,k}\circ \gra \iota_{k,j}
=\id$.

The graded version $\gra \G_{\infty}$ is
also well defined by the requirement that $[f_{v,\mu_\infty}]$ is mapped into
$\psi_{v,\mu_\infty}$ and both $\gra \Gamma_\infty$ and $\gra \Gamma_\infty^{-1}$
are $G_j$-morphisms of rings.

\begin{theorem}\label{th:4.19} Assume that the rank of $\bX_\infty$ is finite. Suppose that $k\ge j$, $\mu\in \Lambda^+$,  $v\in V_{\mu_k}$
and $w\in V_{\mu_j}$. Denote by
$\pr_{V_{\mu_j}}$ the projection $V_{\mu_k}\to V_{\mu_j}$. Then
\begin{equation}\label{eq:prG}
\pr_{j,k}\Gamma_{\mu_k}(f_{v,\mu_k})=\Gamma_{\mu_j}(f_{\pr_{V_{\mu_j}}(v),\mu_j})=\psi_{\pr_{V_{\mu_j}}(v),\mu_j}
\end{equation}
and
\begin{equation}\label{eq:iG}
\iota_{k,j}\Gamma_{\mu_j}(f_{w,\mu_{j}})=
\Gamma_{\mu_k}(f_{w,\mu_k})=\psi_{w,\mu_k}\, .
\end{equation}
In particular $\pr_{j,k}\circ \iota_{k,j}=\id$.  Similar statements hold for the
inverse maps. Let $\gra \G^\infty =\varprojlim \gra \G_j$ and
$\gra \G^{\infty, -1}=\varprojlim \gra \G_j^{-1}$ We therefore have a commutative diagram:

\begin{equation*}
\xymatrix{
\cdots & \ar[l]
\gra \C[\bZ_k] \ar[d]_{\gra\G_k} & \ar[l]^{\pr_{j,k}}
\gra \C[\bZ_j]\ar[d]_{\gra \G_k} & \ar[l] ^{\pr_{k ,\infty}} \varprojlim \gra \C[\bZ_j]\ar[d]_{\gra \G^{\infty}} \\
\cdots &\ar[l]
\C[\Xi_k]\ar[u]_{\gra \G^{-1}_k}& \ar[l]^{\pr_{k,j}}  \C[\Xi_j] \ar[u]_{\gra \G_j^{-1}} &
\ar[l]^{\pr_{k,\infty}} \varprojlim \C[\Xi_j] \ar[u]_{ \gra \G^{\infty, -1}} 
}
\end{equation*}
\end{theorem}

\section{The Radon transform and its dual}\label{se:RT}
\noindent
For the moment we fix the symmetric spaces $\bX$, $\bY$, and $\bZ$ and leave out the index $j$. The Radon transform or its dual is initially defined on the space of compactly supported function. As the dual Radon transform is an integral over the compact group $K_o$ it is well defined on  $\C [\Xi]$ and $\C[\Xi_o]$, the space of regular functions on $\Xi$. But $N_o$ is  noncompact, so the Radon transform cannot be defined on $\C [\bX]$ as an integral over $N_o$. This problem was addressed in \cite{HPV02,HPV03}, and we recall the main results here. Then, based on ideas from \cite{G06,GKO}, we introduce two integral kernels which allow us to express both the Radon transform and the dual Radon transform as an integral against an integral kernel. We start the section by recalling the double fibration transform introduced in  \cite{H66,H70}.

\subsection{The double fibration transform}\noindent
Assume that $G$ is a Lie group and $H$ and $L$ two closed subgroups. We assume that all of those groups as well as $M=H\cap L$ are unimodular.
We have the double fibration

\begin{equation}\label{eq:dFibDi}
 \xymatrix{&G/M \ar[dl]_\pi \ar@{->}[dr]^p&\\
G/H& &G/L }
\end{equation}
\medskip

\noindent
where $\pi$ and $p$ are the natural projections. We say that $x=aH$ and $\xi=bL$ are incident if $aH\cap bL\not=\emptyset$. For $x\in G/H$ and $\xi \in G/L$ we set
\[\hat{x}:=\{\eta \in G/L\mid x \text{ and } \eta \text{ are incident }\}\]
and similarly
\[\xi^\vee :=\{y \in G/H\mid \xi \text{ and } y \text{ are incident }\}\, .\]

Assume that if  $ a\in L$ and $aH\subset HL$, then $a\in H$ and similarly, if $b\in H$ and $bL\subset LH$ then $b\in L$. Then we can view the points in $G/L$ as subsets of $G/H$, and similarly points in $G/H$ can be viewed as subsets of $G/L$. Then $\hat x$ is the set of all $\eta$ such that $x\in \eta$ and $\xi^\vee$ is the set of points $y\in G/H$ such that $y\in \xi$. We also have
\[\hat{x}=p (\pi^{-1}(x)) = aH\cdot \xi_0\simeq H/L \quad\text{and}\quad \xi^\vee =\pi (p^{-1}(\xi )) =b L\cdot x_o \simeq K/L.\]

Fix invariant measures on all of the above groups and the homogeneous spaces $G/H$, $G/L$, $G/M$, $H/M$ and $L /M$ such that for $f\in C_c(G)$ we have
\begin{eqnarray*}
\int_G f(a)\, da &=&\int_{G/M} f(am)\, d\mu_{G/M}(aH\cap L) dm\\
&=&
\int_{G/H} f(ah)\, d\mu_{G/H}(gH)dh\\
&=& \int_{G/L} f(an)\, d\mu_{G/L}(aL)dn
\end{eqnarray*}
and for $f\in C_c(H)$ and $\varphi \in C_c(L)$
\[\int_H f(h)\, dh=\int_{H/M} \int_M f (am)\, dmd\mu_{H/M}(aM)\]
and
\[\int_L f(a)\, da=\int_{L/M} \int_M f (am)\, dmd\mu_{L/M}(aM)\, .\]

The definition of the \textit{Radon transform} and the \textit{dual Radon transform} is now as follows. Let $x_o=eH$ and $\xi_o=eL$. If $\xi=a\cdot \xi_o\in \Xi$ and $x=b\cdot x_o\in X$, then
\begin{equation}\label{eq:RtrDfib}
\hf (\xi ):=\int_{L/M}f (a b\cdot x_o)\, d_{L/M}(bM)\, \quad f\in C_c(G/H)
\end{equation}
and
\begin{equation}\label{eq:DRtrDfib}
\vv (x) := \int_{H/M} \varphi (bh\cdot \xi_o)\, d_{H/M} (hM)\, .
\end{equation}
Then the following  duality holds:
\[\int_{G/L}\hat f (\xi )\varphi (\xi )\, d\mu_{G/L}(\xi ) = \int_{G/H} f(x)  \varphi^\vee
(x)\, d\mu_{G/H} (x)\, .\]

\subsection{The horospherical Radon transform and its dual}\label{se:RandDual}
\noindent
The example  studied most is the case $G=G_o$, $H=K_o$ and $L=M_oN_o$, where we use the notation from the earlier sections. In this case we find the the  \textit{horospherical} Radon transform which from now on we will simply call the Radon transform and its dual. The corresponding integral transforms are
\[\cR f(a\cdot \xi_o)=\hf (a\cdot \xi_o) =\int_{N_o} f(an\cdot x_o)\, dn\, ,\quad f\in C_c(\bX)\]
and
\[\cR^*  \varphi  (b\cdot x_o) =\vv (b\cdot x_o) = \int_{K_o} \varphi (bk\cdot z_o) \, dk\, ,\quad \varphi \in C_c (\Xi_o)\, .\]
Here $dk$ is the invariant probability measure on $K_o$.

As mentioned earlier the dual Radon transform
\[\cR^*\psi (a\cdot x_o)=\int_{K_o} \psi (ak\cdot \xi_o)\, dk\]
is well defined on $\C[\Xi]$. It is clearly a $G$-intertwining operator. Thus, there exists $c_\mu\in \C$ such that
\begin{equation}\label{eq-Radon-Gamma}
\cR^*_\mu:=\cR^*|_{\C[\Xi]_\mu} = c_\mu\Gamma_\mu^{-1}\, ,\quad \mu\in \gL^+.
\end{equation} 
To describe  the evaluation of $c_\mu$ we recall the functions $f_\mu$ and $\psi_\mu$ from Section~\ref{sec1}. We find
\begin{eqnarray*}
\cR^*\psi_{\mu} (a\cdot x_o)&=&\int_{K_o} \du{u}{\pi_{\mu}^* (a)\pi_{\mu}^* (k)u_{\mu}^*}\, dk \\
&=& \du{\pi_{\mu}(a)^{-1}u}{\int_{K_o}\pi_{\mu}^* (k)u_{\mu}^*\, dk}\\
&=&c_\mu f_{\mu}(a\cdot x_o)\, .
\end{eqnarray*}
Thus $c_\mu$ is determined by
\[\int_{K_o} \pi_{\mu}^* (k)u_{\mu}^*\, dk=c_\mu e_{\mu}^*\, .\]

For $g\in G_o$ write $g=k(g)a(g)n(g)$ with $(k(g),a(g),n(g))\in K_o\times A_o\times N_o$. For $\lambda \in \fa^*$ and $a=\exp X\in A_o$ write $a^\lambda =e^{\lambda (X)}$. Let
\[\rho :=\frac{1}{2}\sum_{\alpha\in \Sigma^+}m_\alpha\alpha
=\frac{1}{2}\sum_{\alpha\in \Sigma_0^+}\left(\frac{1}{2} m_{\alpha /2} + m_\alpha \right)
\, \alpha\]
where $m_\alpha :=\dim_\C \fg_\alpha$. We normalize the Haar measure on $\overline{N}_o=\theta (N_o)$ by
\[\int_{\overline{N}_o} a(\bar n)^{-2\rho}\, d\bar n = 1\, .\]
Define for $\lambda \in \fa^*(0)=\{\lambda\in\fa^*\mid (\forall \alpha \in\Sigma^+)\, \, \alpha \ip{\Re (\lambda)}{\alpha} >0\}$
\[\bc (\lambda ):=\int_{\overline{N}_o} a(n)^{-\lambda - \rho}\, d\bar n\, .\]
Then $\bc$ is holomorphic on $\fa^*(0)$. By the Gindikin-Karpelevich formula \cite{GK62} which expresses $\bc$ as a rational function in Gamma-functions depending on the multiplicities $m_\alpha$, the function $\bc$ has a meromorphic extension to all of $\fa^*$. The function $\bc$, which  is called the \textit{Harish-Chandra $c$-function} can be used to calculate the constant $c_\mu$ from \eqref{eq-Radon-Gamma}.

\begin{lemma}\label{le:cMu} Let $\mu\in\gL^+$. Then $c_\mu =\sqrt{\bc (\mu^* +\rho)}= \sqrt{\bc (\mu +\rho)}$.
\end{lemma}

\begin{proof} See \cite[Thm. 3.4]{DOW12} or \cite[Thm 9]{HPV02}.
\end{proof}

We note that the statement in \cite{DOW12} is that $c_\mu = \sqrt{\bc (\mu^* +\rho)}$. On the other hand the result in \cite{HPV02} is that $c_\mu=\sqrt{\bc (\mu +\rho)}$. But the Gindikin-Karpelevich formula implies that
\begin{equation*}
\bc (\lambda )=\bc (-w_o\lambda )\, .
\end{equation*}
As $-w_o\rho =\rho$ it follows that $\sqrt{\bc (\mu^* +\rho)}= \sqrt{\bc (\mu +\rho)}$.

\subsection{The Radon transform as limit of integration over spheres}\label{se:RT-spheres}
We have seen that the dual Radon transform and the normalized transform
$\Gamma_\mu^{-1}$ are the same up to a normalizing factor that depends on the $K$-representation $\mu$. No such relation  exists for $\Gamma_\mu$ and $\cR|_{\C[\bX]_\mu}$ because the definition of the Radon transform $\cR$ does not make sense for regular functions. However, in \cite{HPV03} a solution was proposed by considering the Radon transform as a limit of Radon transform of a double fibrations transforms with both stabilizers being compact. Hence the corresponding integral transforms are well defined for regular functions. We recall the setup from \cite{HPV03}.

Since both $K_o$ and $K_o^a$ are compact, both the Radon transform associated to this double fibration and its dual transform are well-defined
on regular functions and give $G$-equivariant linear maps
\[
\mathcal R_a: \ \C[\bX]  \longrightarrow \C[\Sph{a}\bX] \quad\text{and}\quad
\mathcal R_a^*:  \  \C[\Sph{a}\bX]  \longrightarrow \C[\bX].
\]
One can identify $\Sph{a}\bX$ with $\bX$ associating to each sphere its center. Then $\mathcal R_a$ and $\mathcal R_a^*$ become linear endomorphisms of $\C[\bX]$. By definition, $\mathcal R_a$ is then obtained by integrating over spheres of radius $a$, while $\mathcal R_a^*$ is obtained by integrating over spheres of radius $a^{-1}$.

The spaces $\bX$, $\Xi_o$, and $\Sph{a}\bX$ can all be embedded into the algebraic dual space
\begin{equation*}
\C[\bX]^*=\prod_{\mu \in \Lambda} \C[\bX]_\mu^*,
\end{equation*}
which we equip with the product topology. Let  
$v_\mu^\pm\in \C[\bX]_\mu^*$ be the highest (resp. lowest) weight vector uniquely determined by the normalizing condition 
\begin{equation*}
\du{ f_\mu^\pm} {v_\mu^\mp}=1,
\end{equation*}
where $f_\mu^+:=f_\mu$, and $f_\mu^-:=s_{x_o}\cdot f_{\mu}$.  Then we have
$v_\mu^-= s_{x_o}\cdot v_{\mu^*}^+$ where $s_{x_o}$ is the
symmetry around the base point $x_o$. Similarly, let
$v_\mu ^0\in \Big(\C[\bX]_\mu^*\Big)^{K_o}$ be the $K_o$-invariant vector uniquely determined by the normalizing condition
\begin{equation*}
 \du{f_\mu^0}{v_\mu^0}=1,
\end{equation*}
where $f_\mu^0$ is the unique $K$-invariant function in $\C[\bX]_\mu$ such that $f_\mu^0(x_o)=1$. Note that $f_\mu^0$ is a zonal spherical function.\label{zsf}
The $G_o$-equivariant map $\iota_e\colon \bX\to \C[\bX]^*$ defined by $\langle f,\iota_e(x)\rangle= f(x)$ for  $f\in \C[\bX]$
is injective and satisfies
$\iota_e(o)=(v_\mu^0)_{\mu\in\Lambda^+}$ (see \cite[Section 5]{HPV03}).
For any $a\in A_o$ obtains an injective $G_o$-equivariant map $\iota_a\colon \Sph{a}\bX\to \C[\bX]^*$ by
\begin{equation*}
 \iota_a(\cS_a(x))=\Big(\frac{x_\mu}{a^{\mu^*}}\Big)_{\mu\in\Lambda^+},\qquad \text {if} \quad  \iota_e(x)=(x_\mu)_{\mu\in \Lambda^+}.
\end{equation*}
In particular,
\begin{equation*}
\iota_a(\cS_a)=\Big(\frac{a\cdot v_\mu^0}{a^{\mu^*}}\Big)_{\mu\in\Lambda^+}.
\end{equation*}
The induced map $\iota_a^*\colon \C[\bX]\to\C[\Sph{a}\bX]$ is a $G_o$-module isomorphism. Finally, $\iota(\xi_o):=(v_\mu^+)_{\mu\in \Lambda^+}$ defines a $G$-equivariant map $\iota\colon \Xi_o\to \C[\bX]^*$. Since the stabilizer of $\xi_0$, as well the stabilizer of
$(v_\mu^+)_{\mu \in \Lambda^+}$, is $M_oN_o$, the map $\iota$ is well-defined and injective. The induced map $\iota^*\colon \C[\bX]\to \C[\Xi_o]$ coincides with the $G_o$-module isomorphism $\Gamma$, where we note that $\C[\bX]=\C[\bZ]$ and $\C[\Xi_o]=\C[\Xi]$. We obtain the diagram
\begin{equation*}
\xymatrix@M=7pt{
&G_o/M_o\ar [dl]_{q_a}\ar [dr]^{q}&\\
\Sph {a}\bX \ar [dr]^{\iota_a} && \ar [dl]_{\iota} \Xi_o\\
&\R[\bX]^*&}
\end{equation*}
which turns out to be commutative. This is part of the following proposition which is proven in \cite[Section 6]{HPV03}:

\begin{proposition}\label{prop4_12}
  \begin{enumerate}
    \item[(i)] $\lim_{\alimit}\iota_a\circ q_a=\iota\circ q$.

  \item[(ii)] $\lim_{\alimit}q_a^*\circ \iota_a^*=q^*\circ \iota^*$.

  \item[(iii)] $\lim_{\alimit}\cR_a^*\circ \iota_a^*=\cR^*\circ \iota^*$.

  \item[(iv)] $\cR^*(\iota^*(f))=\mathbf c(\mu+\rho) f$ for all $f\in \C[\bX]_\mu$.

  \end{enumerate}
\end{proposition}

We note that those results can be applied to $\C[\bZ]$ and $\C [\Xi ]$ by restriction and
holomorphic extension.

Suppose now  as before that we have a propagated sequence of symmetric spaces $\bZ_j\to \bZ_k$, $k\ge j$.
We also assume that the rank is finite. Then, on each level,  we have $\iota_j^* =\Gamma_j$. Therefore, we can define for $f\in\C[\bZ_\infty]$
\[\iota^*_\infty (f)=\iota_j^*(f)\quad \text{ if }\quad f\in \C[\bZ_j]\, . \]
Then

\begin{proposition}\label{prop5_4} If rank $\bX_\infty$ is finite and $f\in\C[\bZ_\infty]$. Then
\[\Gamma_\infty f =\iota_\infty^*(f)\, .\]
\end{proposition}
\subsection{The kernels defining the normalized Radon transform and its dual} \label{se:kernel}
We start by stating the following version of the orthogonality relations which are usually formulated in
terms of invariant inner products. The proof for this version as is the same as the usual one.
The invariant measure on $U$ is always the normalized Haar measure. The
following is the usual orthogonality relation stated in form of duality.

\begin{lemma} For  $\mu,\nu\in\gL^+$ let $d (\mu )=\dim V_\mu=\dim V_{\mu^*}$. Then
\[\int_{U} \du{u}{\pi_{\mu}^* (b)u^*}\du {\pi_\nu (a) v }{v^* }\, db=\delta_{\mu,\nu}
d (\mu )^{-1} \,  \du{v}{ u^* }\du{u}{v^*}\]
for all $u,v\in V_\mu$ and for all $u^*,v^*\in V_\mu^*$.
\end{lemma}

It follows by \cite[Thm. 3.4]{DOW12} that $\du{e_\mu}{e_{\mu}^*}=\bc (\mu +\rho)$. Define
\begin{equation}\label{def:kXX}
\kX {a} :=\sum_{\mu\in\gL^+} d (\mu ) \bc (\mu +\rho)
\, \du{u_\mu}{\pi_\mu^* (a) e_{\mu}^*}
\end{equation}
and
\begin{equation}\label{de:kXi}
\kXi {a}:= \sum_{\mu\in\gL^+}  d (\mu )
\, \du{e_\mu}{\pi_\mu^* (a) u_{\mu}^*} \, .
\end{equation}

\begin{lemma}\label{le:4.11} Let $\cO =\{nak \in NAK\mid (\forall j ) \, \, |a^{-\omega_j}|<1\}$. Then $\cO$ is  open in $G$. The set $\cO$ is right $K$-invariant and left $MN$-invariant. Furthermore, the sums defining $k_{\bZ}$ and $k_{\Xi}$ converge uniformly on compact subsets of $\cO$ and define  holomorphic functions on $\cO$.
\end{lemma}

\begin{proof} Write $\mu=k_1\omega_1+\ldots +k_r \omega_r$. Let $x\in \cO$ and write
$b_j=a^{-\omega_j}$. Then $|b_j|<1$ for
$j=1,\ldots ,r$ and
\begin{eqnarray*}
d(\mu  )\du{u_\mu}{\pi_{\mu}^* (x) u_{\mu}^*}&=&
d(\mu)\du{\pi_\mu (k^{-1}a^{-1}n^{-1})u_\mu}{u_{\mu}^*}\\
&=& d(\mu) a^{-\mu}\\
&=&d (\mu)
\prod_{j=1}^r b_j^{k_j}\, .
\end{eqnarray*}

Similarly, we have
\[d(\mu )\bc (\mu +\rho)\du{u_\mu}{\pi_{\mu}^* (x) e_{\mu}^*}=d(\mu )\bc (\mu +\rho)\prod_{j=1}^r b_j^{k_j}\, .\]
The claim follows now because $d(\mu )$ is polynomial in $k_1,\ldots ,k_r$ and $\bc (\mu + \rho)<1$.
\end{proof}

The function $k_{\bZ}$ is left $MN$-invariant and right $K$-invariant. Hence $k_{\bZ}$ can also be viewed as a function on $\Xi\times \bX$ given by
\[\kX{a\cdot \xi_o, b\cdot x_o}:=\kX {a^{-1}b}\, .\]
The function  $k_\Xi$ is left $K$-invariant and right $MN$-invariant and can be viewed as a function $k_\Xi$ on  $\bX\times \Xi$ defined by
\[\kXi {a\cdot x_o,b\cdot \xi_o} := \kXi {a^{-1}b}\, .\]

Even if the sums \eqref{def:kXX} and \eqref{de:kXi} do not in general converge for all $a\in G$, they are well defined as  linear $G$-maps $\C[\bX]\to \C[\Xi]$, respectively $\C[\Xi ]\to \C[\bX]$, given by
\[K_\bZ (f)(\xi) :=\int_U f(u\cdot x_o)\kX{\xi ,u\cdot x_o}\, du,\]
respectively
\[K_\Xi (f )(x) :=\int_U f (u\cdot \xi_o)\kXi {x,u\cdot \xi_o }\, du,\]
where $du$ is the normalized Haar measure on $U$. On the right hand side only finitely many terms are nonzero so the sums converge.
Note that the first integral can be written as an integral over
the compact symmetric space $U/K_o$ and the second integral is an integral over $U/M_o$.

\begin{theorem}\label{th:IntTr} We have $K_{\bZ}=\Gamma$ and $K_{\Xi}=\Gamma^{-1}$.
\end{theorem}
\begin{proof} It is enough to show that $K_{\bZ}|_{\C [\bX]_\nu}=\Gamma_\nu$ and
$K_{\Xi}|_{\C[\Xi]_\nu}=\Gamma_\nu^{-1}$ for all $\nu\in\gL^+$. Thus we have
to show that $K_{\bZ}(f_\nu )=\psi_\nu$ and $K_{\Xi} (\psi_\nu )=f_\nu$.

We have
\begin{eqnarray*}
K_{\bZ} (f_\nu) (a\cdot \xi_o)&=&\sum_{\nu\in\gL^+} \int_U \frac{ d (\mu )}{\du {e_\mu}{e_{\mu}^*}}
\, \du{\pi_\mu (a) u_\mu}{\pi_{\mu}^* (b) e_{\mu}^*}\du{u_\nu}{\pi_{\nu}^*(b)e_{\nu}^*} \, du\\
&=& \du{\pi_{\nu}^* (a) u_{\nu}^*}{u_\nu}
\\
&=&\psi_\mu (a\cdot \xi_o)\, .
\end{eqnarray*}

The statement for $k_\Xi$ is proved in the same way.
\end{proof}

\begin{remark}{\rm The above we realized $\Gamma$ and $\Gamma^{-1}$ as integral operators.
Similar to \cite{G06} one could also  consider the integral operator given by the kernel
$\widetilde{K}(a\cdot \xi_o,b\cdot x_o)=\widetilde{k} (a^{-1}b)$ where
\begin{equation}
\widetilde{k}(a)=\sum_{\mu\in\gL^*}f_\mu (a)
=\sum_{\mu\in\gL^+} \du{u_\mu}{\pi_{\mu}^*(a)e_\mu^*}\, .
\end{equation}
Then we have the following theorem, see also \cite{G06}:
\begin{theorem} Let $\cO$ be as in Lemma \ref{le:4.11} and let $x=kan\in\cO$. be so that $|a^{\omega_j}|<1$ for $j=1,\ldots ,r=\rank \bX$. Then
\[\widetilde k (b)=\prod_{j=1}^r\frac{1}{1-a^{\omega_j}}\]
and $\widetilde k $ is holomorphic on $\cO$.
\end{theorem}
\begin{proof} Let $\mu = k_1\omega_1+\ldots k_r\omega_r\in \gL^+$. For $x=nak\in\cO$ write
as before $b_j=a^{-\omega_j}$. Then
\[\du{u_\mu}{\pi_{\mu^*}(x)e_{\mu^*}}=a^{-\mu}=\prod_{j=1}^r b_j^{k_j}\, .\]
It follows that
\[\widetilde k (x)=\sum_{k_1=0}^\infty b_1^{k_1} \ldots \sum_{k_r=0}^\infty b_r^{k_r}
=\prod_{j=1}^r(1-b_j)^{-1}\]
which finish the proof.
\end{proof}
}
\end{remark}

\subsection{The Radon transform and its dual on the injective limits}\label{se:RInfty}
\noindent
In this section we discuss the extension of the Radon transform and its dual on the infinite dimensional spaces.

The kernels defined in (\ref{def:kXX}) and (\ref{de:kXi}) do not define functions on
$\bZ_\infty$, respectively $\Xi_\infty$, because the changes in the dimensions as we move
from one space to another. But we still have the following:

\begin{theorem}\label{th:KernelInf} Assume that the sequence $\{ \bZ_j\}$ is admissible.  For $f\in \C_i[\bZ_\infty]$ and $\psi \in \C_i[\Xi_\infty]$
the pointwise limits
\[K_{\bZ_\infty} f (\xi )=\lim K_{\bZ_j}f (\xi) \text{ and } K_{\Xi_\infty} \psi (x)=\lim K_{\bZ_j}\psi (x)\]
are well defined and
\[\Gamma_\infty f = K_{\bZ_\infty}f \quad \text{ and }\quad \Gamma_\infty^{-1} \psi
=K_{\Xi_\infty}\psi\, .\]
\end{theorem}

\begin{proof} If $v\in V_{\mu_j}$, then $v\in V_{\mu_k}$ for all $k\ge j$. Hence Theorem
\ref{th:IntTr} implies that if $s>k>j$ are so that $\xi\in \Xi_k$ we have
\[K_{\bZ_s} f (\xi )= K_{\bZ_k}f (\xi)\, .\]
Hence the sequence becomes constant and the claim follows. The argument for
$\bK_{\Xi_\infty} \psi$ is the same.
\end{proof}

Note that the equations (\ref{eq:prG}) and (\ref{eq:iG}) together with Theorem \ref{th:IntTr} imply that the maps $\gra K_{\bZ_\infty}=\varprojlim \gra K_{\bZ_j}$ and $\gra K_{\Xi_\infty}=\varprojlim K_{\Xi_\infty}$ are well defined. Here $\gra$ stands for $\gra K_{\bZ_j}([f])=[K_{\bZ_j}f]$ respectively $\gra K_{\Xi_j}(f) =[K_{\Xi_j}f]$.

\begin{theorem}\label{th:5:19} Assume that the sequence $\bZ_j$ is admissible. Then
\[\gra \G^\infty f =\gra K_{\bZ_\infty}f \quad \text{ and }\quad \gra \G^{\infty ,-1}f=\gra K_{\Xi_\infty}f\, .\]
\end{theorem}

In order to make the results of Section~\ref{se:RT-spheres} useful for limits of symmetric spaces, we first have to extend the notion of a sphere of radius $a$. So let $a=\varinjlim a_j\in A_\infty:= \varinjlim A_j$. We call $A_\infty$ \emph{regular}, if $Z_{K_j}(a_j)=M_j$ for all $j$. This is a useful notion only in the finite rank case, so we will assume for the remainder of this section that we are in the situation of Remark~\ref{rem:fin-rank}.

A simple calculation shows that the diagram
$$\xymatrix{
\bX_j\ar[r]\ar[d]_{\iota_{j,e}} &\bX_k\ar[d]_{\iota_{k,e}}\\
\C[\bX_j]^*\ar[r]&\C[\bX_k]^*
}
$$
of $G_j$-module morphisms is commutative. The normalizations from Section~\ref{se:RT-spheres} are compatible and yield the following commutative diagram
$$\xymatrix{
\Sph{a_j}(\bX_j)\ar[r]\ar[d]_{\iota_{j,a_j}}&\Sph{a_k}(\bX_k)\ar[d]_{\iota_{k,a_k}}\\
\C[\bX_j]^*\ar[r]&\C[\bX_k]^*\\
\Xi_{j,o}\ar[r]\ar[u]^{\iota_{j}}&\Xi_{k,o}\ar[u]^{\iota_{k}}
}
$$
of $G_j$-module morphisms. In fact, for the commutativity of the upper square one uses the equality $a_j^{\mu_j^*}=a_k^{\mu_k^*}=a^{\mu^*}$, whereas the commutativity of the lower square is a consequence of Theorem~\ref{thm:Gamma}. Thus the $\Sph{a_j}(\bX_j)$ and the $\C[\bX_j]^*$ form inductive systems. For the corresponding inductive limits $\Sph{a}(\bX_\infty)$ and $\C[\bX_\infty]^*$ we obtain the commutative diagram
\begin{equation*}
\xymatrix@M=7pt{
&G_{\infty,o}/M_{\infty,o}\ar [dl]_{q_{\infty,a}}\ar [dr]^{q_\infty}&\\
\Sph {a}(\bX_\infty) \ar [dr]^{\iota_{\infty,a}} && \ar [dl]_{\iota_\infty} \Xi_{\infty,o}\\
&\C[\bX_\infty]^*&}
\end{equation*}

Using this notation Proposition~\ref{prop4_12}(i) remains true:

\begin{equation}\label{limit1}
\lim_{a\to \infty}\iota_{\infty,a}\circ q_{\infty,a}=\iota_\infty\circ q_\infty.
\end{equation}

Due to $G$-equivariance, it suffices to prove that
$$\lim_{\alimit}\iota_{\infty,a}(q_a(eM))=\iota_\infty(q_\infty(eM)).$$
This means  
that for fixed $\lambda\in \Lambda^+$ and $j\ge j_o$ we have to verify
$$\lim_{\alimit} \frac{a\cdot v_{j,\lambda}^0}{a^{\lambda^*}}=v_{j,\lambda}^+.$$
Writing $v_{j,\lambda}^0\in \C[\bX_j]_{\lambda}^*$ as a sum of weight vectors (the highest weight being $\lambda^*$), we see that
$$\lim_{\alimit} \frac{a\cdot v_{j,\lambda}^0}{a^{\lambda^*}}=k_{j,\lambda} v_{j,\lambda}^+$$
for some constant $k_{j,\lambda}$. The calculation
\begin{equation}
\frac{\langle f_{j,\lambda}^-, a\cdot v_{j,\lambda}^0\rangle}{a^{\lambda^*}}
=\frac{ (a^{-1}f_{j,\lambda}^-)(x_o)}{a^{\lambda^*}}
=\frac{a^{\lambda^*}f_{j,\lambda}^-(x_o)}{a^{\lambda^*}}=1.
\end{equation}
shows that $k_{j,\lambda}=1$. This implies \eqref{limit1}.

Equation \eqref{limit1} yields immediately the convergence of the induced maps of function spaces.

\begin{equation} \label{limit2}
\lim_{\alimit}q_{\infty,a}^*\circ \iota_{\infty,a}^*=q_\infty^*\circ \iota_\infty^*.
\end{equation}

It was shown in \cite[Thm. 4.7]{DOW12} that the limit
\[\bc_\infty (\mu_\infty ) :=\lim_{j\to \infty} \bc (\mu_j+\rho_j)\]
exists and is strictly positive if the rank of $\bX_\infty$ is finite. Define
\[\cR^* : \C_i[\Xi_\infty]\to \C_i[\bZ_\infty] \]
by
\begin{equation}\label{de:Rinf}\cR^*(f)(x)=\lim_{j\to \infty} \cR^*_jf (x)\, .
\end{equation}
\begin{theorem}\label{th:522} Assume that the rank of $\bX_\infty$ is
finite. Let $f\in \C_i[\Xi_\infty]$. Then the pointwise limit (\ref{de:Rinf}) exists
and for $f\in  \C[\Xi_\infty]_{\mu_\infty}$ we have
\[\cR^* f= \bc_\infty (\mu_\infty)^{1/2}\, \Gamma^{-1}f\quad \text{ and } \quad
\cR^*_\infty (\iota^*(f))=\bc_\infty (\mu_\infty) f\, .\]
\end{theorem}
\begin{proof} As every function in $\C_i[\Xi_\infty]$ is a finite sum of
elements in $\C[\bX_\infty]_{\nu}$ we only have to show this for
fixed $\mu_\infty\in\gL^+_\infty$. But then the claim
follows from   (\ref{eq-Radon-Gamma}), Theorem \ref{th:5:19},
Proposition \ref{prop4_12}, part (iv), and Proposition \ref{prop5_4}.
\end{proof}

As  we are assuming that the rank of $\bX_\infty$ is finite, it follows from \cite{DOW12} that
$\left(\hat{V}_{\mu_\infty}^*\right)^{K_\infty}\not=\{0\}$. Denote by $\pr_\infty$ the
 orthogonal projection
\[\pr_\infty :\hat{ V}_{\mu_\infty}^* \to \left(\hat{V}_{\mu_\infty}^*\right)^{K_\infty}\, .\] 

It follows also from the calculations in \cite{DOW12} that the sequence $\{e_{\mu_j}^*\}$ converges
to $e^*_{\mu_\infty}$ in the Hilbert space $\hat{V}_{\mu_\infty}^*=\hat{V}_{\mu^*_\infty}$. Hence
\[\pr_\infty \left (\hat{ V}_{\mu_\infty}^*\right)= \left(\hat{V}_{\mu_\infty}^*\right)^{K_\infty}
\subset \varprojlim V_{\mu_j}^*\, .\]
Finally, a simple calculation shows that
\begin{equation}\label{eq:pr}
\pr_\infty (w)=\lim_{j\to \infty} \int_{K_j} \pi_{\mu_\infty} (k)w\, dk\,.
\end{equation}

If $f=f_{w,\mu_\infty}\in \C[\bZ_\infty]_{\mu_\infty}$,  then there exists $j_o$ such that
$f|_{\bZ_j}=f_{w,\mu_j}\in \C[\bZ_j]_{\mu_j}$ for all $j\ge j_o$. We have
\begin{equation}\label{eq:ra}
\cR_{a}(f|_{\bZ_j})(g\cdot \cS_a)=\int_{K_j}\du{w}{\pi_{\mu_j}^*(g)\pi_{\mu_j}^*(k)
\pi_{\mu_j}^* (a ) e_{\mu_j}^*}\, dk\, .
\end{equation}

\begin{theorem}\label{th:525} Let $f\in \C_i[\bZ_\infty]$. Then the pointwise limit
\[\cR_{a,\infty}f (g\cdot \cS_a):=\lim_{j\to\infty} \cR_{a,j}f(g\cdot \cS_a)\]
exists and the following holds:
\begin{enumerate}
\item $\cR_{a,\infty}f(g\cdot \cS_a)=\du{w}{\pi_{\mu_\infty}^*(g)\pr_\infty (\pi_{\mu_\infty}^* (a)
e_{\mu_\infty}^*)}$ if $f\in \C_i[\bZ_\infty]_{\mu_\infty}$.
\item $\displaystyle \lim_{a\to \infty} \cR^*_{a,\infty}\circ \iota_a^*=\cR^*_\infty \circ \iota^*$.
\end{enumerate}
\end{theorem}
\begin{proof} This follows from \eqref{eq:pr},  \eqref{eq:ra}, and Proposition \ref{prop4_12}.
\end{proof}

\begin{remark} {\rm We note that the following diagrams do not
commute
$$\xymatrix{
\C[\Xi_j]\ar[r]^{\iota_{k,j}}\ar[d]_{\cR_j^*} &\C[\Xi_k]\ar[d]^{\cR_k^*}\\
\C[\bX_j]\ar[r]_{\iota_{k,j}}&\C[\bX_k]
}
\quad
\xymatrix{
\C[\Sph{a_j}(\bX_j)]\ar[r]^{\iota_{k,j}}\ar[d]_{\cR_{j,a_j}^*} &\C[\Sph{a_k}(\bX_k)]\ar[d]^{\cR_{k,a_k}^*}\\
\C[\bX_j]\ar[r]_{\iota_{k,j}}&\C[\bX_k]
}
$$
This follows from the corresponding commutative diagrams for the normalized
dual Radon transforms $\gamma_j^{-1}$ and the normalizing factor that relates those two transforms. This
makes the corresponding theory for infinite rank spaces problematic as in that case
$\lim_{j\to\infty} \bc (\mu_j+\rho_j)=0$.} \phantom{m}\hfill $\qed$
\end{remark}

\end{document}